\documentclass[11pt]{article}
\usepackage{amsthm, amssymb, amsmath}
\usepackage{verbatim, float}
\usepackage{mathrsfs}
\usepackage{graphics}
\usepackage[usenames,dvipsnames]{pstricks}
\usepackage{epsfig}
\usepackage{pst-grad} 
\usepackage{pst-plot} 
\usepackage{cases}
\usepackage{algorithmic}
\usepackage{bm}

\usepackage[top=2cm, bottom=2cm, left=2.5cm, right=2.5cm]{geometry}
\usepackage{url}

\usepackage{hhline}

\usepackage{etoolbox}
\makeatletter
\patchcmd{\@makecaption}
  {\scshape}
  {}
  {}
  {}
\makeatother

\usepackage[singlelinecheck=false]{caption}
\usepackage{diagbox}

\theoremstyle{plain}
\newtheorem{theorem}{Theorem}
\newtheorem{lemma}{Lemma}

\newtheorem{corollary}{Corollary}
\newtheorem{proposition}{Proposition}

\theoremstyle{definition}
\newtheorem{definition}{Definition}
\newtheorem{example}{Example}

\newcommand{\B}{{\mathcal B}}

\newcommand{\D}{{\mathcal D}}

\newcommand{\X}{{\mathcal X}}

\renewcommand{\O}{{\mathcal O}}

\newcommand{\lam}{\lambda}

\newcommand{\define}{\stackrel{\mbox{\tiny $\triangle$}}{=}}

\newcommand{\nin}{\noindent}

\newcommand{\mis}{{\min}_{\sum}} 
\newcommand{\mas}{{\max}_{\sum}}
\newcommand{\misd}{{\min}^*_{\sum}} 
\newcommand{\ds}{\Delta_{\sum}}
\newcommand{\rs}{r_{\sum}}
\newcommand{\stsn}{{\sf{STS}}(n)}
\newcommand{\sts}{{\sf{STS}}}
\newcommand{\xb}{(\mathcal{X},\mathcal{B})}
\renewcommand{\sb}{{\sf{sum}}(B)}
\newcommand{\pb}{\pi_B}
\newcommand{\ps}{\pi_S}
\newcommand{\ob}{\oplus_{b}}
\newcommand{\os}{\oplus_{s}}
\newcommand{\lb}{\ell_B}
\newcommand{\ls}{\ell_S}

\begin{document}

\title{MaxMinSum Steiner Systems\\ for Access-Balancing in Distributed Storage}
\author{Hoang Dau and Olgica Milenkovic\\
Coordinated Science Laboratory, University of Illinois at Urbana-Champaign \\
Emails: \{hoangdau, milenkov\}@illinois.edu}
\date{}
\maketitle


\begin{abstract} Many code families such as low-density parity-check codes, fractional repetition codes, batch codes and private information retrieval codes with low storage overhead rely on the use of combinatorial block designs or derivatives thereof. 
In the context of distributed storage applications, one is often faced with system design issues that impose additional constraints on the coding schemes, and therefore on the underlying block designs. 
Here, we address one such problem, pertaining to server access frequency balancing, by introducing a new form of Steiner systems, termed MaxMinSum Steiner systems. MaxMinSum Steiner systems are characterized by the property that the minimum value of the sum of points (elements) within a block is maximized, or that the minimum sum of block indices containing some fixed point is maximized. 
We show that proper relabelings of points in the Bose and Skolem constructions for Steiner triple systems lead to optimal MaxMin values for the sums of interest; for the duals of the designs, we exhibit block labelings that are within a $3/4$ multiplicative factor from the optimum. 
\end{abstract}

\section{Introduction}

Due to their unique combinatorial features, Steiner systems have found many applications in constructive coding theory, ranging from low-density parity-check code design~\cite{vasic2004combinatorial} to distributed storage~\cite{el2010fractional,silberstein2015optimal} to batch codes~\cite{silberstein2016optimal} and low-redundancy private information retrieval~\cite{fazeli2015codes}. In many such applications, one is faced with system constraints that impose additional restrictions on the underlying point-block incidence structures. One such constraint arises in the area of coding for distributed storage~\cite{dimakis2010network}, and pertains to access balancing of servers. Although issues such as delay-storage tradeoffs, volume (load) balancing, and chunk allocation for distributed storage have been studied in depth~\cite{leong2012distributed,joshi2014delay}, the equally relevant issue of access balancing seems to have been overlooked. Unlike the well known load-balancing paradigm, access balancing does not aim to perform near-uniform distribution of data content on the available servers. Instead, it aims to balance the access requests to the disks by using file or file chunk popularity information~\cite{cherkasova2004analysis}. 
The approach is to estimate the popularity of the files or different chunks of the files, and then store files on servers in such a way that each server has roughly the same average access frequency. Clearly, access balancing based on popularity may lead to load disbalances, as a small number of files may be highly popular and hence highly accessed and most appropriately stored alone on a server in order to minimize access collisions. The most desirable solution is to balance both the load and the average access frequency of servers, which is the topic of this contribution. In particular, the focus is on load and access balancing of storage schemes that utilize codes based on Steiner systems and duals of Steiner systems.

The notion of access balancing is best illustrated on the example of Fractional Repetition Codes (FRCs)~\cite{el2010fractional}, which combine Maximum Distance Separable (MDS) and repetition codes with a carefully chosen placement strategy (The probabilistic version of these codes that adopts a random placement strategy is known as DRESS codes~\cite{Pawar_etal2011}). In such a distributed storage system, redundant data is stored on $N$ storage nodes so that one can recover the whole data content from any $K<N$ nodes. Furthermore, the system is designed in such a way that when a node fails, the node may be repaired by creating a new node and by contacting $D$ survivor nodes and downloading (part of) their content to the new node; the parameter $D$ is often referred to as the repair degree. The restored distributed storage system is required to maintain the original MDS and repetition coding properties. In addition, one often requires the failed node recovery procedure to have the \emph{exact repair property}, where the newly added replacement node has the same content as the failed node, and the \emph{minimum-bandwidth regenerating property}, which allows for a low-complexity repair/download process. In the particular setting of FRCs, the replacement node downloads exactly one chunk from $D$ survivor nodes without any processing or additional coding/decoding. An example of a small-scale FRC is shown in Fig.~\ref{fig:balancing}, in the upper half of the diagram. The code construction is straightforward: User information is parsed into chunks, and the chunks are subsequently encoded using an MDS code. This scheme, according to Fig.~\ref{fig:balancing} involves ten chunks and five blocks. The ten chunks are repeated a certain number of times (twice in the given example), and then placed in groups so that no two servers share more than one chunk (four chunks on each server in the given example). A solution to this grouping strategy is shown in the right hand corner of Fig.~\ref{fig:balancing}. If, for example, node $1$ storing the chunks $\{1,2,3,4\}$ fails at some point of time, a replacement node is created which contacts $D=4$ nodes $2,3,4,$ and $5$ to recover the same chunks stored by node $1$. An FRC minimum-bandwidth storage repair system with the properties above may be constructed via Steiner systems or their duals, in which the stored content of a node equals the indices of the blocks containing that node as a point.

\begin{figure}[t]
\centering
\includegraphics[scale=1]{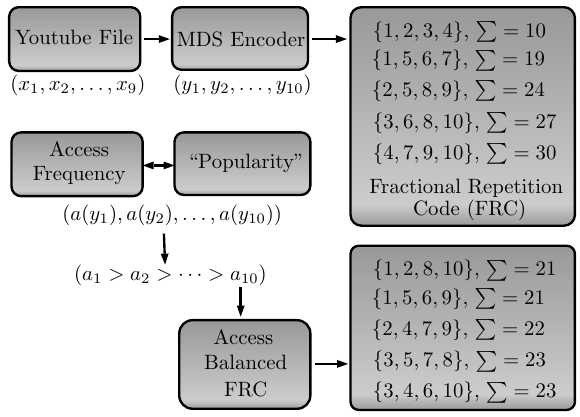}
\caption{An example illustrating the notions of an FRC and access balancing for FRCs. Access frequencies, which are proportional to the popularity scores of the nine data chunks and the parity chunk, are denoted by $a(y_1),a(y_2),\ldots,a(y_{10})$. For simplicity, we assume a strict ordering of the popularities of the data chunks and assign the lowest popularity to the parity chunk. Consequently, in the example, the chunk labels $\{1,2,\ldots,10\}$ reflect their popularities - $10$ represents the lowest popularity chunk, while $1$ represents the highest popularity chunk. Alternatively, the chunk labeled $1$ has the highest access frequency, while the chunk labeled $10$ has the lowest access frequency.}
\label{fig:balancing}
\end{figure}
Assume now that the labels of the chunks are directly proportional to their popularities and hence, their access frequencies. Then, the overall popularity of the files stored on a server may be approximated by the sum of the labels of the chunks stored on the server.
In the example, the popularities of server information differ widely: The smallest overall popularity score equals $10$, while the largest equals $30$. 
To address this disbalance issue, consider the chunk placement shown at the bottom right-hand side of Fig.~\ref{fig:balancing}. The blocks of data chunks still satisfy the property that each pair of chunks appears at most once on any server, but the variation of the average popularity values is significantly smaller than that for the previous setting: The smallest popularity score is $21$, while the largest popularity score is $23$. The popularity, and hence access variation of the placement, went down from $200\%$ to $10\%$. The data chunk assignment is dictated by what we refer to as a \emph{MaxMinSum placement}, which maximizes the minimum sum of chunk popularities on the servers. This placement also turns out to minimize the maximum difference between sums of chunk popularities. The MaxMinSum FRC scheme also ensures natural load balancing in terms of the volume of the data stored, as each server stores the same number of data chunks.

We conclude this example by noting that the main system design parameter to consider when dealing with file popularities is access rate. If file chunks are ranked according to their popularity, their access rates have been shown experimentally to follow a Zipf law, which asserts that the $k$-th most popular chunk has access frequency $1/k^\alpha$, $\alpha>0$ (see, for example~\cite{BreslauWebCaching} and the many follow up papers on this line of work). To ensure accurate server access balancing, one would need to convert popularities into to access rates, and then assign weights to the points in the design equal to their access rates. This is clearly a much more elaborate approach which obscures the underlying mathematical problems, and instead we chose to use the simple mapping described in the example.

The contributions of our work are two-fold. First, we introduce the notion of access balancing in distributed storage systems that use coded information. Second, we propose new access balancing techniques for FRCs based on a new family of combinatorial designs that satisfy constraints on the sum of values within the blocks or sum of labels of blocks containing a certain point. Despite all definitions and concepts being valid for Steiner systems and block designs in general, the focus of the work is on Steiner triple systems (STSs),
and in particular, MaxMinSum (MinMaxSum) Steiner triple systems for which the goal is to maximize (minimize) the minimum (maximum) block-sum in the design. We also describe the related questions of minimizing the difference of block-sums or the ratio of the maximum and minimum block-sums, but relegate their detailed analysis to future work. 

Our main findings are bounds on the value of minimum and maximum block-sums and a proof that shows that for all Steiner triple systems constructed using the Bose and Skolem methods, one can exhibit a point relabeling that achieves the maximum minimum block-sum or minimum maximum block-sum. The results may be extended to the case of duals of STS, which are of equally significant importance in distributed storage applications. 

The paper is organized as follows. Section~\ref{sec:prelim} contains the problem statement and derivations of bounds on various constrained Steiner system sum values. Section~\ref{sec:labels-STS} contains constructive proofs for the existence of MaxMinSum Steiner triple systems for all parameter values for which they exist. These results are extended to duals of Steiner triple systems in Section~\ref{sec:labels-dual}, where it is shown that one may exhibit dual placements only a factor $3/4$ away from the optimal bound. Concluding remarks and open problems are given in Section~\ref{sec:conclusions}.

\section{Problem Statement and Preliminaries} \label{sec:prelim}

We start by defining some basic concepts from design theory and by introducing the new concept of MaxMinSum and related designs. 
The interested reader is referred to~\cite{colbourn2006handbook} for an in-depth treatment of the general subject of combinatorial designs and Steiner systems in particular.

\begin{definition} 
\label{def:design}
A $t$-$(n,k,\lam)$ \emph{design} is a pair $\xb$ where $\X$, the point set, is an $n$-set and $\B$, the block set, is 
a collection of $k$-subsets of $\X$ (blocks) such that every $t$-subset of $\X$ is contained in precisely $\lam$ blocks. 
A $t$-$(n,k,1)$ design is often referred to as a \emph{Steiner system} and denoted by $S(t,k,n)$. A $2$-$(n,3,1)$ design is called a \emph{Steiner triple system} (STS) of order $n$, denoted by $\stsn$. 
\end{definition}  

The parameters of a block design satisfy two basic constraints:
\begin{equation*}
|\B| \times k= n \times r,
\end{equation*}
where $r$ denotes the number of blocks containing a given point, and
\begin{equation*}
\lambda \times \binom{n-1}{t-1}= r \times \binom{k-1}{t-1},
\end{equation*}
which for the case of an $\stsn$ reduce to $|\B|=n(n-1)/6$ and $r = (n-1)/2$.

\begin{definition} 
\label{def:minsum}
For each block $B \subseteq \X$ let $\sb$ denote the block-sum $\sum_{x \in B}x$. Given a $t$-$(n,k,\lam)$ design $\xb$, we introduce the following quantities. 
\begin{itemize}
	\item The min-sum of the design, defined as $\mis(\B) \define \min_{B \in \B} \sb$,
	\item The max-sum of the design, defined as $\mas(\B) \define \max_{B \in \B} \sb$,
	\item The difference-sum of the design, defined as $\ds(\B) \define \mas(\B)-\mis(\B)$,
	\item The ratio-sum of the design, defined as $\rs(\B) \define \mas(\B) / \mis(\B)$.
\end{itemize}
\end{definition} 
For the purpose of access-balancing, the most suitable performance metrics are the difference-sum and ratio-sum. Unfortunately, questions pertaining to these types of designs are also the most difficult ones to analyze. We therefore establish bounds on all four 
metrics of interest, but mostly focus on upper bounds for the min-sum and lower bounds on the max-sum and constructions of 
designs that meet these bounds. We then proceed to show that STS constructions that meet the min-sum bound also offer \emph{order-optimal} difference sum and ratio-sum values.

For integers $a \leq b$, let $[a,b]$ denote the set $\{{a,a+1,\ldots,b\}}$. 
Furthermore, for $|\X| = n$, let $\X = [0,n-1]$ unless stated otherwise. 
An upper bound on the min-sum of a Steiner system is presented in Proposition~\ref{pro:upperbound}. 
The proof follows standard arguments in combinatorial design theory.

\begin{proposition}[Upper Bound on the Min-Sum]
\label{pro:upperbound}
The min-sum of any Steiner system $S(t,k,n)$ satisfies the inequality
\[
\mis \leq \dfrac{n(k-t+1)+k(t-2)}{2}.
\] 
\end{proposition}
\begin{proof} 
Suppose that $\xb$ is a Steiner system $S(t,k,n)$, and $\X = [0,n-1]$. 
Let $\B_{[0,t-2]}$ be the set of all blocks of the system that contain $[0,t-2]$ as a subset.
As $\xb$ is a $t$-$(n,k,1)$ design, the sets $B \setminus [0,t-2]$, $B \in \B_{[0,t-2]}$, partition the set $[t-1,n-1] = \X \setminus [0,t-2]$. 
Hence, $|\B_{[0,t-2]}| = (n-t+1)/(k-t+1)$. Moreover, if $B_{\min}$ is a block in 
$\B_{[0,t-2]}$ with smallest block-sum, then 
\[
\begin{split}
{\sf{sum}}(B_{\min}) &\leq \dfrac{1}{\left|\B_{[0,t-2]}\right|}\sum_{B \in \B_{[0,t-2]}} \sb\\ 
&= (1 + \cdots + (t-2)) + \dfrac{(n+t-2)(n-t+1)/2}{(n-t+1)/(k-t+1)}
= \dfrac{nk-n(t-1)+k(t-2)}{2}. 
\end{split}
\]
The first inequality asserts that the smallest block-sum is bounded from above by the average block-sum, while the 
equality takes into account that the points in $[0,t-2]$ belong to all blocks in $\B_{[0,t-2]}$. This proves the claimed result. 
\end{proof} 
Using a proof that follows along the same lines, with $\B_{[0,t-2]}$ replaced by $\B_{[n-t+1,n-1]}$, where $\B_{[n-t+1,n-1]}$ is the set of all blocks containing $[n-t+1,n-1]$, we arrive at the following result.

\begin{proposition}[Lower Bound on the Max-Sum] \label{pro:lowerbound}
The max-sum of any Steiner system $S(t,k,n)$ satisfies the inequality
\[
\mas \geq \dfrac{nk+nt-n-kt}{2}.
\] 
\end{proposition}

\begin{corollary} 
\label{cr:upperbound}
The min-sum of any $\stsn$ is at most $n$, while its max-sum is at least $2n-3$. Every $\stsn$ with min-sum $n$ must contain $(n-1)/2$ blocks of the form $\{0,i,n-i\}$, $i \in [1,(n-1)/2]$. Moreover, every $\stsn$ with max-sum $2n-3$ must contain $(n-1)/2$ blocks of the form $\{j,n-2-j,n-1\}$, $j \in [0,(n-3)/2]$.  
\end{corollary} 
\begin{proof} 
The first statement follows directly from Propositions~\ref{pro:upperbound} and~\ref{pro:lowerbound}, with $k = 3$ and $t = 2$. 
Furthermore, from the proofs of these proposition, if an $\stsn$ has min-sum $n$ then every block containing $0$ must have the same block-sum as the average, i.e. a block-sum equal to $n$. Since there are precisely $(n-1)/2$ such blocks, they must be of the form $\{0,i,n-i\}$, $i \in [1,(n-1)/2]$. The same line of reasoning implies that every $\stsn$ with max-sum $2n-3$ must contain $(n-1)/2$ blocks of the form $\{j,n-2-j,n-1\}$, $j \in [0,(n-3)/2]$.
\end{proof} 

Note that bounds on the min-sum and max-sum of $\stsn$ may also be derived by considering the block containing the points $0$ and $1$, and the block containing the points $n-2$ and $n-1$, which must have sums at most $n$ and at least $2n-3$, respectively. Unfortunately, for general values of $k$ and $t$, this simple argument gives weaker bounds than those presented in Proposition~\ref{pro:upperbound} and Proposition~\ref{pro:lowerbound}.

Combining the results for the min-sum and max-sum, one can easily derive bounds on the difference-sum and the ratio-sum of an $\stsn$ of the form: 
$$\ds = \mas - \mis \geq n - 3$$ 
and 
$$\rs = \mas/\mis \geq 2 - 3/n.$$ 
However, a more careful analysis of the blocks containing \emph{either} $0$ \emph{or} $n-1$ yields stronger bounds on the difference-sum and the ratio-sum of an STS. 

\begin{proposition}
\label{pro:dsrs_lowerbound}
The difference-sum and ratio-sum of any $\stsn$ satisfy $\ds \geq n$ and $\rs \geq 2$. 
\end{proposition}
\begin{proof} 
Let 
$$\B_0 = \left\{(0,x_i,y_i) \colon i \in [1,(n-1)/2], \; x_i,y_i \in [1,n-1], x_i < y_i\right\}$$ 
and 
$$\B_{n-1} = \left\{(z_i,t_i,n-1) \colon i \in [1,(n-1)/2], \; z_i,t_i \in [0,n-2], z_i < t_i \right\}$$ 
be the sets of all $(n-1)/2$ blocks of the system containing $0$ and $n-1$, respectively. Note that the pairs $\{x_i,y_i\}$, $i \in [1,(n-1)/2]$, partition the set $[1,n-1]$, and the pairs $\{z_i,t_i\}$, $i \in [1,(n-1)/2]$, partition the set $[0,n-2]$.
Let $B_{0,n-1} = \{0,p,n-1\}$, $p \in [1,n-2]$, be the unique block in the system containing both $0$ and $n-1$. Then, the min-sum of the system cannot exceed the average block-sum of $\B_0 \setminus B_{0,n-1}$, that is
\begin{equation} 
\label{eq:b0}
\mis \leq \left\lfloor \dfrac{1}{\frac{n-3}{2}}\sum_{B \in \B_0 \setminus \{B_{0,n-1}\}} \sb \right\rfloor = 
\left\lfloor\dfrac{2}{n-3}\Bigg(\binom{n}{2} - (n-1+p)\Bigg)\right\rfloor = \left\lfloor\dfrac{(n-1)(n-2)-2p}{n-3}\right\rfloor. \notag
\end{equation} 
Moreover, the max-sum of the system must be greater than or equal to the average
block-sum of $\B_{n-1} \setminus B_{0,n-1}$, that is
\begin{equation} 
\label{eq:b1}
\begin{split}
\mas \geq \left\lceil\dfrac{1}{\frac{n-3}{2}}\sum_{B \in \B_{n-1} \setminus \{B_{0,n-1}\}} \sb \right\rceil &= 
n-1 + \left\lceil\dfrac{2}{n-3}\Bigg(\binom{n-1}{2} - p\Bigg)\right\rceil \\ 
&= n-1 + \left\lceil\dfrac{(n-1)(n-2)-2p}{n-3}\right\rceil. \notag
\end{split}
\end{equation} 
Therefore, we have 
\begin{equation}
\label{eq:ds_lowerbound} 
\ds = \mas - \mis \geq (n - 1) + \left\lceil\dfrac{(n-1)(n-2)-2p}{n-3}\right\rceil
- \left\lfloor\dfrac{(n-1)(n-2)-2p}{n-3}\right\rfloor. 
\end{equation} 
Clearly, if $n-3$ does not divide $(n-1)(n-2)-2p$ then \eqref{eq:ds_lowerbound} implies that $\ds \geq n$. Now suppose that $n-3$ does divide $(n-1)(n-2)-2p$, which holds if and only if $p = 1$ or $p = n-2$. We consider these two cases separately. 

\textbf{Case 1.} $p = 1$. Then $B_{0,n-1} = \{0,1,n-1\}$. Let $B_{0,2} = \{0,2,x\}$
and $B_{n-2,n-1} = \{y,n-2,n-1\}$ be the unique blocks containing the pairs $\{0,2\}$ and $\{n-2,n-1\}$, respectively. Since the pair $\{0,n-1\}$ does not belong to $B_{0,2}$, we have $x \leq n-2$. Similarly, since the pairs $\{0,n-1\}$ and $\{1,n-1\}$ do not belong to $B_{n-2,n-1}$, we have $y \geq 2$. If $x \leq n-3$ then $\mis \leq {\sf{sum}}(B_{0,2}) \leq n-1$ and $\mas \geq {\sf{sum}}(B_{n-2,n-1}) \geq 2n-1$, which implies that $\ds \geq n$. If $x = n-2$ then $y \geq 3$ because the pair $(2,n-2)$ now belongs only
to block $B_{0,2}$. In this case, $\mis \leq {\sf{sum}}(B_{0,2}) = n$ and $\mas 
\geq {\sf{sum}}(B_{n-2,n-1}) \geq 2n$, which again implies that $\ds \geq n$. 

\textbf{Case 2.} $p = n-2$. Then $B_{0,n-1} = \{0,n-2,n-1\}$ and $B_{0,1} = \{0,1,x\}$ for $x \leq n-3$. If $x \leq n - 4$, then $\mis \leq {\sf{sum}}(B_{0,1}) \leq n-3$. By Corollary~\ref{cr:upperbound}, $\mas \geq 2n-3$. Therefore, in this case, we have $\ds \geq n$. If $x = n-3$, then $B_{0,1} = \{0,1,n-3\}$ and $\mis \leq {\sf{sum}}(B_{0,1}) = n-2$.
Let $B_{n-3,n-1} = \{y,n-3,n-1\}$. Since neither of the two pairs $\{0,n-1\}$ and $\{1,n-3\}$ belongs to $B_{n-3,n-1}$, we deduce that $y \geq 2$, which leads to $\mas \geq {\sf{sum}}(B_{n-3,n-1}) \geq 2n-2$. Thus, $\ds \geq (2n-2) - (n-2) = n$.

We have hence established that $\ds \geq n$. Based on this bound on the difference-sum, we can prove that $\rs \geq 2$ as follows. By Corollary~\ref{cr:upperbound}, $\mis \leq n$. Moreover, since $\mas - \mis \geq n$, we have
\[
\rs = \dfrac{\mas}{\mis} \geq \dfrac{\mis + n}{\mis} = 1 + \dfrac{n}{\mis} \geq 2, 
\]
which establishes the claimed bound $\rs \geq 2$. 
\end{proof} 

Note that in general, the bounds established in Proposition~\ref{pro:dsrs_lowerbound} are not tight. As an example, for $n = 7$, an exhaustive computer search over all possible permutations of seven points in the unique $\sts(7)$ yields the minimum ratio-sum $15/7 > 2$. For $n = 13$, the search performed on two nonisomorphic $\sts(13)$ establishes that the minimum difference-sum equals $14 > 13$. For $n = 7, 9$, the lower bound on the difference-sum is met, while for $n = 9$, the lower bound on the ratio-sum is met as well. 

\section{Existence of Steiner Triple Systems with Maximum Min-Sum} \label{sec:labels-STS}

In what follows, we focus our attention on constructing $\stsn$ that have min-sums attaining the upper bound stated in Corollary~\ref{pro:upperbound}. We refer to such $\stsn$ as MaxMinSum $\stsn$. Note that by replacing each block $B = \{x,y,z\}$ in a set system $(\X=[0,n-1],\B)$ with the block $B^* = \{n-1-x,n-1-y,n-1-z\}$, one arrives at a new set system $(\X,\B^*)$ with $\mas(\B^*) = 3n-3-\mis(\B)$. Therefore, a construction that produces an $\stsn$ with min-sum equal to $n$ immediately gives rise to an $\stsn$ with max-sum equal to $2n-3$, which also achieves the lower bound stated in Corollary~\ref{cr:upperbound}.

The idea behind our approach is to consider known constructions for $\stsn$ and for each of them find a permutation (relabeling) of the set $\X$ that ensures that every permuted block has a sum greater than or equal to $n=|\X|$. Generally, given an arbitrary $\stsn$, one needs to examine $n!$ permutations, which can be performed in a reasonable amount of time by computer search only for small values of $n$, say $n \leq 13$. Corollary~\ref{cr:upperbound}, however, reduces the search space to $n\big((n-1)/2\big)!\, 2^{(n-1)/2}$. Indeed, one can consider $n$ possible ways to map a point $x \in \X$ to $0$. For each choice of $x$, there are $\big((n-1)/2\big)! \, 2^{(n-1)/2}$ ways to create a one-to-one map between the set $\{B \in \B \colon x \in B\}$ and the set $\big\{\{0,i,n-i\} \colon i \in [1,(n-1)/2]\big\}$. The choice of $x$ and of a map described above uniquely determines the corresponding permutation on $\X$. This reduction in complexity makes a computer search possible for $n \leq 19$, which is helpful in the search for a general solution. 

It was established by Kirkman~\cite{Kirkman1847} that an $\stsn$ exists if and only if $n \equiv 1,3 \pmod n$. Two direct constructions were established by Bose~\cite{Bose1939}, for $n \equiv 3 \pmod 6$, and by Skolem~\cite{Skolem1958} and Hanani~\cite{Hanani1960}, for $n \equiv 1 \pmod 6$. We first present a special case of the Bose Construction (see, for instance~\cite[p.~4]{LindnerRodger1997} or~\cite[p.~127]{Stinson2004}) and describe how to permute (relabel) the points of the resulting system to achieve a maximum value for the min-sum.\\ 

\textbf{The Bose Construction.}
Let $n = 3m \geq 9$, where $m$ is an odd number, so that $n \equiv 3 \pmod 6$. 
We define a binary operation $\ob$ on the set $[0,m-1]$ as follows. 
For $0 \leq x, y \leq m-1$,
\begin{equation} 
\label{eq:binop}
x \ob y \define \frac{m+1}{2}(x+y) \pmod m.
\end{equation} 
The point set equals $\X' = \left\{(x,i) \colon x \in [0,m-1], i \in [0,2]\right\}$. 
The block set $\B'$ of an ${\sf{STS}}(3m)$ is the union of the following two types of blocks.
\begin{itemize}
	\item Type 1: $B_x = \left\{ (x,0), (x,1), (x,2) \right\}$, for every $x \in [0,m-1]$.
	\item Type 2: $B_{x,y,i} = \left\{ (x,i), (y,i), (x\ob y, i + 1 \pmod 3)\right\}$, for every $0 \leq x < y \leq m - 1$ and $i \in [0,2]$. 
\end{itemize}

\begin{figure}[H]
\centering
\includegraphics[scale=1]{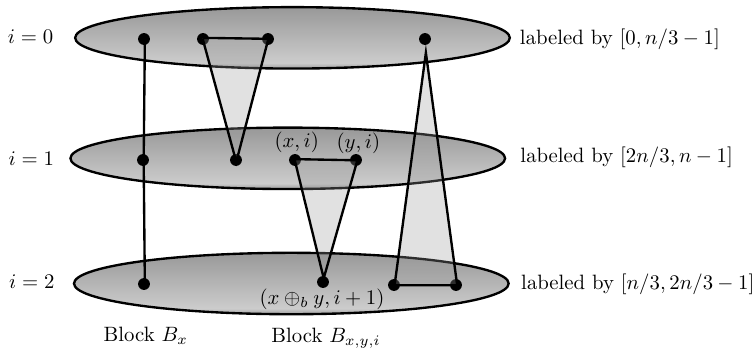}
\caption{Illustration of Bose Construction of a Steiner triple system of order $n$ and the intuition for our point labeling that achieves min-sum $n$.}
\label{fig:BoseIllustration}
\end{figure}

The Bose Construction is illustrated in Fig.~\ref{fig:BoseIllustration}. In order to achieve a min-sum of $n$, 
the general rule is to label $(x,0)$, $(x,1)$, and $(x,2)$ with numbers in $[0,n/3-1]$, $[2n/3-1,n-1]$, and $[n/3,2n/3-1]$, respectively, or using any cyclic reordering of these sets.
According to this rule, it is immediate that the block-sum of $B_x$ is at least $0 + 2n/3 + n/3 = n$ for every $x \in [0,m-1]$. 
Using labels in $[0,n/3-1]$ for $(x,0)$ and labels in $[n/3,2n/3-1]$ for $(x,1)$, for instance, would not allow us to achieve the min-sum value of $n$ because the block $B_{x,y,0},$ where $(x,0)$ and $(y,0)$ are labeled by $0$ and $1$, respectively, has a block-sum of at most $2n/3 < n$.
Before formally presenting an optimal labeling, referred to as the Bose Mapping, we provide some intuition behind our approach. 
First, a natural label for $(x,0)$ is $x$. Second, let us consider $x = 0$, $y > 0,$ and assume that ${\sf{sum}}(B_{0,y,0}) \geq n$. Since the labels of $(0,0)$ and $(y,0)$ are $0$ and $y$, respectively, the inequality implies that the label for $(0\ob y,1)$ is at least $n-y$. In fact, we will assign the label $n-y$ to $(0\ob y,1)$ in the Bose Mapping. The labeling rule for $(x,2)$ is more subtle and harder to intuitively justify.    

\textbf{The Bose Mapping.} We introduce a mapping $\pb$, which maps the point set $\X'$ of the $\stsn$ obtained in the Bose Construction to $\X = [0,n-1]$ as follows. For every $x \in [0,m-1]$,
\begin{equation}
\label{eq:Bose}
\begin{split}
(x,0) &\mapsto x,\\
(x,1) &\mapsto \begin{cases} 2m, &\text{ if } x = 0,\\ 
	n - y, &\text{ if } x = 0\ob y \neq 0,\end{cases}\\
(x,2) &\mapsto \begin{cases} m, &\text{ if } x = 0,\\ 
	m + \big(0\ob (m-x)\big), &\text{ if } x \neq 0.\end{cases}
\end{split}
\end{equation}
Applying the Bose Mapping to every point of the system $(\X',\B')$ produced by the Bose Construction, we obtain another $\stsn$ $(\X = [0,n-1], \B)$.

\begin{example}
\label{ex:Bose}
Let $n = 9$, $m = 3$.
The blocks of an $\sts(9)$ obtained from the Bose Construction are given
below. 
\begin{align*}
\text{Type 1: } &\{(0,0),(0,1),(0,2)\},\quad \{(1,0),(1,1),(1,2)\},\quad \{(2,0),(2,1),(2,2)\},\\
\text{Type 2: } &\{(0,0),(1,0),(2,1)\},\quad \{(0,1),(1,1),(2,2)\},\quad \{(0,2),(1,2),(2,0)\},\\
&\{(0,0),(2,0),(1,1)\},\quad \{(0,1),(2,1),(1,2)\},\quad \{(0,2),(2,2),(1,0)\},\\
&\{(1,0),(2,0),(0,1)\},\quad \{(1,1),(2,1),(0,2)\},\quad \{(1,2),(2,2),(0,0)\}.
\end{align*}
According to \eqref{eq:Bose}, the Bose Mapping $\pb$ results in the assignments 
\begin{align*}
(0,0) \mapsto 0,\quad (1,0) \mapsto 1,\quad (2,0) \mapsto 2,\\
(0,1) \mapsto 6,\quad (1,1) \mapsto 7,\quad (2,1) \mapsto 8,\\
(0,2) \mapsto 3,\quad (1,2) \mapsto 4,\quad (2,2) \mapsto 5.
\end{align*}
For example, $(0,1) \mapsto 2m=6$, while 
$$(1,2) \mapsto m + \big(0\ob (m-1)\big) =3+ (4  \pmod 3)=4.$$

Applying this mapping to the blocks above, we obtain the blocks of the output design. 
\begin{align*}
\text{Type 1: } &\{0,6,3\},\quad \{1,7,4\},\quad \{2,8,5\},\\
\text{Type 2: } &\{0,1,8\},\quad \{6,7,5\},\quad \{3,4,2\},\\
&\{0,2,7\},\quad \{6,8,4\},\quad \{3,5,1\},\\
&\{1,2,6\},\quad \{7,8,3\},\quad \{4,5,0\}.
\end{align*}
One can verify that this $\sts(9)$ has min-sum $9$, and that the $(n-1)/2=4$ blocks containing $0$ all have sum equal to $9$. One can also notice that the difference and the ratio of the maximum and minimum sums are $9$ and $2$, respectively, which achieve the lower bounds stated in Proposition~\ref{pro:dsrs_lowerbound}. In general, the Bose Construction
coupled with the Bose Mapping achieves the maximum min-sum (Theorem~\ref{thm:Bose}), but does not necessarily meet the bounds on the difference-sum and the ratio-sum (Proposition~\ref{pro:Bose_dsrs}).  
\end{example}

\begin{theorem} 
\label{thm:Bose}
For $n \geq 9$, $n \equiv 3 \pmod 6$, the Bose Construction coupled with the Bose Mapping produces an $\stsn$ with min-sum equal to $n$. 
\end{theorem} 

We present next a series of lemmas needed to prove Theorem~\ref{thm:Bose}. 
The first two statements are obvious from the definition of the operation $\ob$ and their proofs are hence omitted. 

\begin{lemma} 
\label{lem:Bxy}
If $0 \leq x,y,r \leq m-1$ and $x+y \equiv r \pmod m$, then $x \ob y = 0 \ob r$. 
\end{lemma} 

\begin{lemma} 
\label{lem:B0y}
For every $r \in [0,m-1]$, it holds that
\[
0\ob r = \begin{cases} r/2, &\text{ if } r \text{ is even,}\\ (r+m)/2, &\text{ if } r \text{ is odd.} \end{cases}
\]
\end{lemma} 

The following lemma shows that the Bose Mapping is indeed one-to-one.

\begin{lemma} 
\label{lem:Bbijection}
The Bose Mapping $\pb$ is a bijection from $\X' = \left\{(x,i) \colon x \in [0,m-1], i \in [0,2]\right\}$ to $[0,n-1]$. 
Moreover, 
\begin{itemize}
	\item $\pb((x,0)) \in [0,m-1]$, 
	\item $\pb((x,2)) \in [m,2m-1]$, 
	\item $\pb((x,1)) \in [2m,3m-1]$. 
\end{itemize}
\end{lemma} 
\begin{proof} 
The second statement follows from the definition of $\pi_B$;
noting that $0 \ob 0 = 0$, and hence if $x = 0 \ob y \neq 0$ then $y \neq 0$ as well, which in turn implies that $n-y \leq n-1= 3m-1$.

In order to prove the first statement, it suffices to show that $\pb((x,i)) \neq \pb((x',i))$ whenever $x \neq x'$, for every $i \in [0,2]$. This is straightforward to do since from the definition of the binary operation $\ob$ given in \eqref{eq:binop}, we always have $a \ob b \neq a \ob c$ if $b \neq c$, for $a,b,c \in [0,m-1]$. 
\end{proof} 

The next lemmas state that the sum of every block in the new system $\big(\X = \pb(\X'),\B = \pb(\B')\big)$ is at least $n$. 

\begin{lemma} 
\label{lem:Btype1}
Let $\pb$ be the Bose Mapping. Then for every $x \in [0,m-1]$ we have 
\[
\pb((x,0)) + \pb((x,1)) + \pb((x,2)) \geq n.
\]
\end{lemma} 
\begin{proof} 
From Lemma~\ref{lem:Bbijection},
\[
\pb((x,0)) + \pb((x,1)) + \pb((x,2)) \geq 0 + 2n/3 + n/3 = n.
\]
This proves the claimed result. 
\end{proof} 

\begin{lemma} 
\label{lem:Btype2:1}
Let $\pb$ be the Bose Mapping. Then for every $0 \leq x < y \leq m-1$ we have 
\[
\pb((x,0)) + \pb((y,0)) + \pb((x\ob y,1)) \geq n.
\]
\end{lemma} 
\begin{proof} 
If $x = 0$ then
\[
\pb((0,0)) + \pb((y,0)) + \pb((0\ob y,1)) = 0 + y + (n-y) = n. 
\]
Suppose next that $x > 0$. Set $r \equiv x+y \pmod m$, where $r \in [0,m-1]$. Then by Lemma~\ref{lem:Bxy} we have
\[
\begin{split}
\pb((x,0)) + \pb((y,0)) + \pb((x\ob y,1)) &= x + y + \pb((0\ob r,1))\\
&= x + y + (n-r) \geq x + y + (n-x-y) = n.
\end{split} 
\]
This completes the proof. 
\end{proof} 

\begin{lemma} 
\label{lem:Btype2:2}
Let $\pb$ be the Bose Mapping. Then for every $0 \leq x < y \leq m-1$ we have 
\[
\pb((x,1)) + \pb((y,1)) + \pb((x\ob y,2)) > n.
\]
\end{lemma} 
\begin{proof} 
From Lemma~\ref{lem:Bbijection}, we have $\pb((x,1)) \geq 2m$ and
$\pb((y,1)) \geq 2m$. Therefore, the stated inequality holds.
\end{proof} 

\begin{lemma} 
\label{lem:Btype2:3}
Let $\pb$ be the Bose Mapping. Then for every $0 \leq x < y \leq m-1$ we have 
\[
\pb((x,2)) + \pb((y,2)) + \pb((x\ob y,0)) > n.
\]
\end{lemma} 
\begin{proof} 
If $x = 0$, by Lemma~\ref{lem:B0y} we have
\[
\begin{split}
\pb((0,2)) &+ \pb((y,2)) + \pb((0\ob y,0)) 
= m + \big(m + 0\ob(m-y)\big) + 0\ob y\\
&= 2m + 
\begin{cases} (m-y)/2, &\text{ if } m-y \text{ is even}\\ (m-y+m)/2, &\text{ if } m-y \text{ is odd}
\end{cases}
+
\begin{cases} y/2, &\text{ if } y \text{ is even}\\ (y+m)/2, &\text{ if } y \text{ is odd}
\end{cases}\\
&= \begin{cases} 2m + (m-y)/2 + (y+m)/2, &\text{ if } y \text{ is odd}\\ 2m + (m-y+m)/2 + y/2, &\text{ if } y \text{ is even}
\end{cases}\\
&= n,
\end{split}
\]
where the third equality follows because $m$ is odd.

Suppose that $x > 0$. Let $r \equiv x + y \pmod m$, where $r \in [0,m-1]$. By Lemma~\ref{lem:B0y} and Lemma~\ref{lem:Bxy} we have
\begin{multline*}
\pb((x,2)) + \pb((y,2)) + \pb((x\ob y,0)) 
= \big(m + 0\ob (m-x)\big) + \big( m + o\ob (m-y) \big) + x\ob y\\
= 2m + 
\begin{cases} (m-x)/2, &\text{ if } m-x \text{ is even}\\ (m-x+m)/2, &\text{ if } m-x \text{ is odd}
\end{cases}
+
\begin{cases} (m-y)/2, &\text{ if } m-y \text{ is even}\\ (m-y+m)/2, &\text{ if } m-y \text{ is odd}
\end{cases}\\
+
\begin{cases} r/2, &\text{ if } r \text{ is even}\\ (r+m)/2, &\text{ if } r \text{ is odd}
\end{cases}.
\end{multline*}
Since $m = n/3$ is odd, it is impossible that the three quantities $m-x$, $m-y$, and $r$ are all even. Therefore, 
\[
\begin{split}
\pb((x,2)) + \pb((y,2)) + \pb((x\ob y,0)) &\geq 2m + (m-x)/2 + (m-y)/2 + r/2 + m/2\\
&= n + (r+m-x-y)/2 \geq n. 
\end{split}
\]
This completes the proof of the lemma. 
\end{proof} 

The proof of Theorem~\ref{thm:Bose} now follows from Lemmas~\ref{lem:Bbijection},~\ref{lem:Btype1},~\ref{lem:Btype2:1},~\ref{lem:Btype2:2}, and ~\ref{lem:Btype2:3}. 

We next prove a proposition that establishes bounds on the max-sum, and hence difference-sum and ratio-sum of the Bose Construction coupled with the Bose Mapping. 

\begin{proposition}
\label{pro:Bose_dsrs}
For $n \geq 9$, $n \equiv 3 \pmod 6$, the Bose Construction coupled with the Bose Mapping produces an $\stsn$ with max-sum at most $8n/3-4$. 
Therefore, the difference-sum and the ratio-sum of this construction and mapping are $\O(5n/3)$ and $\O(8/3)$, respectively.
\end{proposition}
\begin{proof} 
From Lemma~\ref{lem:Bbijection}, it is straightforward to see that a block with the maximum sum after applying the Bose Mapping $\pi_B$ must be of the form 
$B = \{(x,1),(y,1),(x\ob y, 2)\}$. Clearly,
\[
\sb \leq (3m-1) + (3m-2) + (2m-1) = 8n/3 - 4. 
\]
As the min-sum is known to be $n$ by Theorem~\ref{thm:Bose}, the claims regarding the difference-sum and the ratio-sum follow easily. 
\end{proof} 

Heuristic evidence suggests that for sufficiently large $n$, the max-sum of the $\stsn$ generated by the Bose Construction coupled with the Bose Mapping equals $8n/3-9$.

We now turn our attention to (a special case of) the Skolem Construction (see, for instance~\cite[p.~9]{LindnerRodger1997} or~\cite[p.~128]{Stinson2004}) when $n \equiv 1 \pmod 6$.\\ 

\textbf{The Skolem Construction.}
Let $n = 3m+1 \geq 7$, where $m$ is an even number, so that $n \equiv 1 \pmod 6$. 
We define a binary operation $\os$ on the set $[0,m-1]$ as follows. 
For $0 \leq x, y \leq m-1$,
\begin{equation} 
\label{eq:Sbinop}
x \os y \define 
\begin{cases}
\dfrac{x+y \pmod m}{2}, &\text{ if } x+y \pmod m \text{ is even},\\
\dfrac{\big(x+y \pmod m\big)+m-1}{2}, &\text{ if } x+y \pmod m \text{ is odd}.
\end{cases}
\end{equation} 
The point set is $\X' = \{\infty\} \cup \left\{(x,i) \colon x \in [0,m-1], i \in [0,2]\right\}$. 
The block set $\B'$ of an ${\sf{STS}}(3m+1)$ is the union of the following three types of blocks.
\begin{itemize}
	\item Type 1: $B_x = \left\{ (x,0), (x,1), (x,2) \right\}$, for every $x \in [0,m/2-1]$.
	\item Type 2: $B_{x,y,i} = \left\{ (x,i), (y,i), (x\os y, i + 1 \pmod 3)\right\}$, for every $0 \leq x < y \leq m - 1$ and $i \in [0,2]$.
	\item Type 3: $B_{x,i} = \left\{ \infty, (x+m/2,i), (x, i + 1 \pmod 3)\right\}$, for every $x \in [0,m/2 - 1]$ and $i \in [0,2]$.
\end{itemize}

Compared to the Bose Construction, the Skolem Construction uses an additional type of blocks that contain the special point $\infty$. Similar reasoning to that used for the Bose Construction, along with some additional considerations, may be used to find an optimal labeling for the Skolem $\stsn$. Here, we assign to the points $(x,0)$, $(x,1)$, and $(x,2)$ labels from $\left[0,\frac{n-1}{3}-1\right]$, $\left[\frac{2(n-1)}{3}+1,n-1\right]$, and $\left[\frac{n-1}{3}+1,\frac{2(n-1)}{3}\right]$, respectively, but reserve the label $\frac{n-1}{3}$ for $\infty$. If one assigns the label $\frac{n-1}{3}-1-x$ to $(x,0)$ then the condition
${\sf{sum}}(B_{x,m-1,0}) \geq n$ implies the corresponding label for $(x,1)$. The full description of this labeling, referred to as the \emph{Skolem Mapping}, is given below. 

\textbf{The Skolem Mapping.} We introduce a mapping $\ps$, which maps the point set $\X'$ of the $\stsn$ obtained in the Skolem Construction to $\X = [0,n-1]$ as follows. For every $x \in [0,m-1]$,
\begin{equation}
\label{eq:Skolem}
\begin{split}
(x,0) &\mapsto m-1-x,\\
(x,1) &\mapsto 
\begin{cases} 2m+1, &\text{ if } x = m/2-1,\\ 
2m+2+y, &\text{ if } x = (m-1) \os y \neq m/2-1,
\end{cases}\\
(x,2) &\mapsto m+1+0\os x,\\ 
\infty &\mapsto m. 
\end{split}
\end{equation}
Applying the Skolem Mapping to every point of the system $(\X',\B')$ produced by the Skolem Construction results in a new $\stsn$ system $(\X = [0,n-1], \B)$.

\begin{example}
\label{ex:Skolem}
Let $n = 13$, $m = 4$.
The blocks of an $\sts(13)$ obtained from the Skolem Construction are given
below. 
\begin{align*}
\text{Type 1: } &\{(0,0),(0,1),(0,2)\},\quad \{(1,0),(1,1),(1,2)\},\\
\text{Type 2: } &\{(0,0),(1,0),(2,1)\},\quad \{(0,1),(1,1),(2,2)\},\quad \{(0,2),(1,2),(2,0)\},\\
&\{(0,0),(2,0),(1,1)\},\quad \{(0,1),(2,1),(1,2)\},\quad \{(0,2),(2,2),(1,0)\},\\
&\{(0,0),(3,0),(3,1)\},\quad \{(0,1),(3,1),(3,2)\},\quad \{(0,2),(3,2),(3,0)\},\\
&\{(1,0),(2,0),(3,1)\},\quad \{(1,1),(2,1),(3,2)\},\quad \{(1,2),(2,2),(3,0)\},\\
&\{(1,0),(3,0),(0,1)\},\quad \{(1,1),(3,1),(0,2)\},\quad \{(1,2),(3,2),(0,0)\},\\
&\{(2,0),(3,0),(2,1)\},\quad \{(2,1),(3,1),(2,2)\},\quad \{(2,2),(3,2),(2,0)\},\\
\text{Type 3: } &\{\infty,(2,0),(0,1)\},\quad \{\infty,(3,0),(1,1)\},\quad \{\infty,(2,1),(0,2)\},\\
&\{\infty,(3,1),(1,2)\},\quad \{\infty,(2,2),(0,0)\},\quad \{\infty,(3,2),(1,0)\}.
\end{align*}
According to \eqref{eq:Skolem}, the Skolem Mapping $\ps$ makes the following assignments: 
\[
\begin{aligned}
(0,0) &\mapsto 3,\quad &(1,0) \mapsto 2,\quad &(2,0) \mapsto 1,\quad &(3,0)\mapsto 0,\text{ }\\
(0,1) &\mapsto 11,\quad &(1,1) \mapsto 9,\quad &(2,1) \mapsto 12,\quad &(3,1)\mapsto 10,\\
(0,2) &\mapsto 5,\quad &(1,2) \mapsto 7,\quad &(2,2) \mapsto 6,\quad &(3,2)\mapsto 8,\text{ } \\
\infty &\mapsto 4. & & &
\end{aligned}
\]
Applying this mapping to the blocks above, we obtain the blocks of the output design. 
\begin{align*}
\text{Type 1: } &\{3,11,5\},\quad \{2,9,7\},\\
\text{Type 2: } &\{3,2,12\},\quad \{11,9,6\},\quad \{5,7,1\},\\
&\{3,1,9\},\quad \{11,12,7\},\quad \{5,6,2\},\\
&\{3,0,10\},\quad \{11,10,8\},\quad \{5,8,0\},\\
&\{2,1,10\},\quad \{9,12,8\},\quad \{7,6,0\},\\
&\{2,0,11\},\quad \{9,10,5\},\quad \{7,8,3\},\\
&\{1,0,12\},\quad \{12,10,6\},\quad \{6,8,1\},\\
\text{Type 3: } &\{4,1,11\},\quad \{4,0,9\},\quad \{4,12,5\},\\
&\{4,10,7\},\quad \{4,6,3\},\quad \{4,8,2\}.
\end{align*}
One can easily verify that this $\sts(13)$ has min-sum $13$, while the max-sum value equals $30$. Hence, the ratio-sum equals $2.31$, which exceeds the lower bound presented in Proposition~\ref{pro:dsrs_lowerbound}. 
Similarly, the difference-sum is $17 > 13$, which also exceeds the lower bound
on the difference-sum. In general, similar to what was established for the Bose Construction, the difference-sum of the Skolem Construction coupled with the Skolem Mapping is $\O(5n/3)$ and the ratio-sum is $\O(8/3)$ (Proposition~\ref{pro:Skolem_dsrs}). 
\end{example}

\begin{theorem} 
\label{thm:Skolem}
For $n \geq 7$, $n \equiv 1 \pmod 6$, the Skolem Construction coupled with the Skolem Mapping produces an $\stsn$ with min-sum equal to $n$. 
\end{theorem}

We make use of several lemmas to prove Theorem~\ref{thm:Skolem}. The first three follow in a straightforward manner from the definition of $\os$, and their proofs are hence omitted. 

\begin{lemma} 
\label{lem:S0y}
For every $y \in [0,m-1]$, it holds that
\[
0\os y = \begin{cases} y/2, &\text{ if } y \text{ is even,}\\ (y+m-1)/2, &\text{ if } y \text{ is odd.} \end{cases}
\]
\end{lemma} 

\begin{lemma} 
\label{lem:Smy}
For every $y \in [0,m-1]$, it holds that
\[
(m-1)\os y = \begin{cases} m-1, &\text{ if } y = 0,\\
(y-1)/2, &\text{ if } y \text{ is odd,}\\ (y+m-2)/2, &\text{ if } y\neq 0 \text{ is even.} \end{cases}
\]
\end{lemma} 

\begin{lemma} 
\label{lem:Sxy}
If $0 \leq x,y\leq m-1$ and $0 \leq r < m-1$, where $x+y \equiv r \pmod m$, then $x \os y = (m-1) \os (r+1)$. 
\end{lemma} 

We now show that the Skolem Mapping is indeed one-to-one.

\begin{lemma} 
\label{lem:Sbijection}
The Skolem Mapping $\pb$ is a bijection from $\X' = \{\infty\} \cup \left\{(x,i) \colon x \in [0,m-1], i \in [0,2]\right\}$ to $[0,n-1]$. Moreover,  
\begin{itemize}
	\item $\ps((x,0)) \in [0,m-1]$, 
	\item $\ps(\infty) = m$,
	\item $\ps((x,2)) \in [m+1,2m]$,
	\item $\ps((x,1)) \in [2m+1,3m]$.
\end{itemize}
\end{lemma} 
\begin{proof} 
The second statement follows from the definition of $\ps$; 
noting that since $m/2-1 = (m-1)\os(m-1)$, if $x = (m-1) \os y \neq m/2-1$ then $y \neq m-1$, which in turn implies that $\ps((x,1)) = 2m+2+y \leq 3m = n - 1$.  

It remains to show that $\ps((x,i)) \neq \ps((x',i))$ whenever $x \neq x'$, for every $i \in [0,2]$. This is also straightforward because from the definition of the binary operation $\os$ given in \eqref{eq:Sbinop}, we always have $a \os b \neq a \os c$ if $b \neq c$, for all $a,b,c \in [0,m-1]$. 
\end{proof} 

The next lemmas state that the sum of every block in the system $\big(\X = \pb(\X'),\B = \pb(\B')\big)$ is at least $n$. 

\begin{lemma} 
\label{lem:Stype1}
Let $\ps$ be the Skolem Mapping. Then, for every $x \in [0,m/2-1]$ we have 
\[
\ps((x,0)) + \ps((x,1)) + \ps((x,2)) > n.
\]
\end{lemma} 
\begin{proof} 
By Lemma~\ref{lem:Sbijection}, 
\[
\ps((x,0)) + \ps((x,1)) + \ps((x,2)) \geq 0 + (m+1) + (2m+1) = n + 1 > n,
\]
which proves the claimed result.
\end{proof} 

\begin{lemma} 
\label{lem:Stype2:1}
Let $\ps$ be the Skolem Mapping. Then for every $0 \leq x < y \leq m-1$ we have 
\[
\ps((x,0)) + \ps((y,0)) + \ps((x\os y,1)) \geq n.
\]
\end{lemma} 
\begin{proof}
First, if $x\os y = m/2-1,$ then $\ps((x\os y,1)) = 2m + 1$. Moreover, as $x < y \leq m-1$, this implies that $x + y = m - 2$. Therefore,
\[ 
\begin{split}
\ps((x,0)) + \ps((y,0)) + \ps((x\os y,1)) &= (m-1-x)+(m-1-y)+(2m+1)\\
&= n + m - 2 - (x+y) = n. 
\end{split}
\]
Second, if $x\os y = m-1$, then $x + y = m-1$. Therefore, 
\[
\begin{split}
\ps((x,0)) + \ps((y,0)) + \ps((x\os y,1)) &= (m-1-x) + (m-1-y) + (2m+2)\\
&= n + m-1 - (x+y) = n.
\end{split} 
\]
Lastly, suppose that $x \os y \notin \{m/2-1,m-1\}$. Let $r \in [0,m-1]$ so that $x + y \equiv r \pmod m$. Then $r < m-1$. By Lemma~\ref{lem:Sxy},
\[
\begin{split}
\ps((x,0)) + \ps((y,0)) + \ps((x\os y,1)) &= (m-1-x) + (m-1-y) + \ps(((m-1)\os(r+1),1))\\
&= (m-1-x) + (m-1-y) + (2m+2+r+1)\\ 
&= n + m+r - (x+y) \geq n.
\end{split}
\]
This completes the proof. 
\end{proof} 

\begin{lemma} 
\label{lem:Stype2:2}
Let $\ps$ be the Skolem Mapping. Then, for every $0 \leq x < y \leq m-1$ we have 
\[
\ps((x,1)) + \ps((y,1)) + \ps((x\os y,2)) > n.
\]
\end{lemma} 
\begin{proof} 
From Lemma~\ref{lem:Sbijection}, we have $\ps((x,1)) \geq 2m+1$ and
$\ps((y,1)) \geq 2m+1$. Therefore, it is clear that the stated inequality holds.
\end{proof} 

\begin{lemma} 
\label{lem:Stype2:3}
Let $\ps$ be the Skolem Mapping. Then, for every $0 \leq x < y \leq m-1$ we have 
\[
\ps((x,2)) + \ps((y,2)) + \ps((x\os y,0)) > n.
\]
\end{lemma} 
\begin{proof} 
Let $r \equiv x + y \pmod m$, where $r \in [0,m-1]$. We have
\begin{align*}
&\ps((x,2)) + \ps((y,2)) + \ps((x\os y,0)) 
= (m+1+0\os x)+ (m+1+0\os y) + (m-1-x\os y)\\
&= n + 
\begin{cases} x/2, &\hspace{-0.25cm}\text{if } x \text{ is even}\\ (x+m-1)/2, &\hspace{-0.25cm}\text{if } x \text{ is odd}
\end{cases}
+
\begin{cases} y/2, &\hspace{-0.25cm}\text{if } y \text{ is even}\\ (y+m-1)/2, &\hspace{-0.25cm}\text{if } y \text{ is odd} 
\end{cases}
-
\begin{cases} r/2, &\hspace{-0.25cm}\text{if } r \text{ is even}\\ (r+m-1)/2, &\hspace{-0.25cm}\text{if } r \text{ is odd}
\end{cases}\\
&\geq n, 
\end{align*}
where the last equality follows because $x+y \geq r$, and $x+y$ and $r$ have the same parity, as $m$ is even.
\end{proof} 

\begin{lemma} 
\label{lem:Stype3:1}
Let $\ps$ be the Skolem Mapping. Then, for every $0 \leq x \leq m/2-1$ we have 
\[
\ps(\infty) + \ps((x+m/2,0)) + \ps((x,1)) \geq n.
\]
\end{lemma} 
\begin{proof} 
If $x = m/2-1,$ then 
\[
\ps(\infty) + \ps((x+m/2,0)) + \ps((x,1)) = m + \big(m-1-(m-1)\big) + 2m + 1 = n.
\]
Now suppose that $x \neq m/2-1$. Then there exists a unique $y \in [0,m-1]$ such that $x = (m-1)\os y$. Since $(m-1)\os 0 = m-1 > m/2-1 \geq x$, we deduce that $y \neq 0$. Then, by Lemma~\ref{lem:Smy},
\[
\begin{split}
\ps(\infty) + \ps((x+m/2,0)) + \ps((x,1)) 
&= m + \Big(m-1-\big((m-1)\os y+m/2\big)\Big) + (2m + 2 + y)\\
&= n + m/2 + y - (m-1)\os y\\
&= n + m/2 + y - \begin{cases} (y-1)/2, &\text{ if } y \text{ is odd}\\
(y+m-2)/2, &\text{ if } y \text{ is even}
\end{cases}\\
&> n.
\end{split}
\]
This completes the proof.
\end{proof} 

\begin{lemma} 
\label{lem:Stype3:12}
Let $\ps$ be the Skolem Mapping. Then, for every $0 \leq x \leq m/2-1$ and for every $i \in [0,1]$, we have 
\[
\ps(\infty) + \ps((x+m/2,i)) + \ps((x,i+1\pmod 3)) \geq n.
\]
\end{lemma} 
\begin{proof} 
By Lemma~\ref{lem:Sbijection}, we have
\[
\ps(\infty) + \ps((x+m/2,i)) + \ps((x,i+1\pmod 3)) 
\geq m + \ps((\cdot,1))
\geq m + (2m+1) = n,
\]
which establishes the claim. 
\end{proof} 

\begin{lemma} 
\label{lem:Stype3:3}
Let $\ps$ be the Skolem Mapping. Then, for every $0 \leq x \leq m/2-1$, we have 
\[
\ps(\infty) + \ps((x+m/2,2)) + \ps((x,0)) > n.
\]
\end{lemma} 
\begin{proof} 
One can show that
\[
\begin{split}
\ps(\infty) + \ps((x+m/2,2)) + &\ps((x,0)) 
\geq m + \big(m+1+0\os (x+m/2)\big) + (m-1-x)\\
&= n-(x+1) + 
\begin{cases} (x+m/2)/2, &\text{ if } (x+m/2) \text{ is even}\\
(x+m/2+m-1)/2, &\text{ if } (x+m/2) \text{ is odd}
\end{cases}\\
&\geq n,
\end{split}
\]
where the last inequality holds because $x \leq m/2-1$. 
This completes the proof. 
\end{proof} 

The proof of Theorem~\ref{thm:Skolem} follows from Lemmas~\ref{lem:Sbijection},~\ref{lem:Stype1},~\ref{lem:Stype2:1},~\ref{lem:Stype2:2},~\ref{lem:Stype2:3},~\ref{lem:Stype3:12}, and~\ref{lem:Stype3:3}. 

The following proposition establishes the max-sum, and hence difference-sum and ratio-sum of the Skolem Construction coupled with the Skolem Mapping. The results are similar to those shown for the Bose Construction.  

\begin{proposition}
\label{pro:Skolem_dsrs}
For $n \geq 7$, $n \equiv 1 \pmod 6$, the Skolem Construction coupled with the Skolem Mapping produces an $\stsn$ with max-sum at most $(8n-11)/3$. 
Therefore, the difference-sum and the ratio-sum of this construction are $\O(5n/3)$ and $\O(8/3)$, respectively.
\end{proposition}
\begin{proof} 
From Lemma~\ref{lem:Sbijection}, it is straightforward to see that a block with the maximum sum after applying the Skolem Mapping $\pi_S$ must be of the form 
$B = \{(x,1),(y,1),(x\ob y, 2)\}$. Clearly,
\[
\sb \leq 3m + (3m-1) + 2m = (8n-11)/3. 
\]
As the min-sum is known to be $n$ by Theorem~\ref{thm:Skolem}, the claims regarding the difference-sum and the ratio-sum hold as claimed. 
\end{proof} 

Heuristic evidence suggests that the max-sum of the $\stsn$ obtained from the Skolem Construction coupled with the Skolem Mapping equals $(8n-17)/3$ for sufficiently large $n$.

\section{Dual Min-Sum of Steiner Triple Systems} \label{sec:labels-dual}

In what follows, we focus our attention on duals of $\stsn$ and establish min-sum results analogous to those established for $\stsn$. It turns out that the min-sum of the duals is only $3/4\times$ of the average block-sum in the STS, 
whereas the min-sum of the STSs is roughly $2/3\times$ of the average block-sum in the STS. Thus, one can 
achieve better balancing properties with the duals than with the STSs themselves. 

The dual of a design is formally defined next. Note that the block size of the dual is precisely the repetition number of the original design~\cite[p.~370]{BoseNairPBIBD}.

\begin{definition}
The dual of a $t$-$(n,k,\lam)$ design $\D = (\X,\B)$ is $\D^*=(\B,\X)$, where $B \in \B$ is contained in $x \in X$ if $x \in B$ in $\D$.  
\end{definition}

Consider the dual of an $\stsn$ where the blocks in $\B$ are labeled from $0$ to $n(n-1)/6-1$.
Our goal is to study the \emph{dual min-sum}, that is, the min-sum of the dual, denoted by $\misd(\B)$. 
A trivial upper bound on the dual min-sum is the averaged dual min-sum. 

\begin{lemma} 
\label{lem:upper_bound_dual}
For any $\stsn$ $(\X,\B)$ with any block ordering, it holds that
\[
\misd(\B) \leq \dfrac{1}{24}(n-1)(n-3)(n+2).
\]
\end{lemma} 
\begin{proof} 
Since each block in $\B$, labeled by a number ranging from $0$ to $n(n-1)/6-1$, contains exactly three points from $\X$, 
on average, the sum of labels of blocks containing a particular point equals 
\[
\dfrac{1}{n} \times \sum_{l = 0}^{n(n-1)/6-1} 3l = \dfrac{1}{24}(n-3)(n-1)(n+2).
\]
As the dual min-sum cannot exceed the average, the lemma follows. 
\end{proof} 

We show next that the $\stsn$ generated by the Bose and by the Skolem Construction, \emph{together with a suitable block ordering}, achieve min-sums that are at most $3/4\times$ away from the upper bound of Lemma~\ref{lem:upper_bound_dual}. 

\begin{theorem}
\label{thm:Bose_YXI}
The Bose $\stsn$, in which the blocks are labeled from $0$ to $n(n-1)/6-1$ in the order: $B_x$, $x = m-1,m-2\ldots,0$, followed by $B_{x,y,i}$, $y = 1,2,\ldots,m-1$, $x = 0,1,\ldots,y-1$, $i = 0,1,2$, has dual min-sum equal to
\[
f_{BYXI}(n) = 
\begin{cases}
\dfrac{55}{1728}n^3 + \dfrac{1}{192}n^2 - \dfrac{9}{64}n - \dfrac{31}{64},& \text{ if } n/3 \equiv 1 \pmod 4,\ n \geq 27,\\\\
\dfrac{55}{1728}n^3 + \dfrac{1}{192}n^2 - \dfrac{13}{64}n + \dfrac{13}{64},& \text{ if } n/3 \equiv 3 \pmod 4,\ n \geq 33.
\end{cases}
\] 
We refer to the specified block labeling as the Bose YXI-labeling. 
\end{theorem}

The dual min-sums of the Bose $\stsn$ with the Bose YXI-labeling specified in Theorem~\ref{thm:Bose_YXI} equal $20$, $104$, and $291$ for $n = 9, 15,$ and $21$, respectively. As there is only one $\sts(9)$ up to isomorphism~\cite{colbourn1999triple}, and all possible $12!$ block-permutations of the Bose $\sts(9)$ produce dual min-sums not exceeding $20$, we conclude that when $n = 9$, the Bose Construction together with the Bose YXI-labeling indeed produces an $\sts(9)$ with maximum dual min-sum. As $n$ goes to infinity, the dual min-sum of the Bose $\stsn$ with the Bose YXI-labeling is very close to the upper bound given in Lemma~\ref{lem:upper_bound_dual}. More precisely, the dual min-sum is only a fraction of $55/72 > 3/4$ away from the upper bound. 

To prove Theorem~\ref{thm:Bose_YXI}, we compute the sum of the labels of all blocks containing an arbitrary pair $(z,i)$, $z \in [0,m-1]$ and $i \in [0,2]$, and then show that this sum is always greater than or equal to a polynomial in $n$. To this end, we make use of a couple of auxiliary lemmas.
The proof of the first lemma is obvious and hence omitted. 

\begin{lemma} 
\label{lem:Bose_Bx}
In the Bose YXI-labeling, the labels of the blocks $B_x$, $x \in [0,m-1]$, equal
\[
\lb(B_x) = m - 1 - x. 
\]
\end{lemma} 

\begin{lemma} 
\label{lem:Bose_Bxyi}
In the Bose YXI-labeling, the labels of the blocks $B_{x,y,i}$, $0 \leq x < y \leq m - 1$, $i \in [0,2]$, equal
\[
\lb(B_{x,y,i}) = m + i + 3x + \dfrac{3}{2}y(y-1),
\]
where $m = n/3$. 
\end{lemma} 
\begin{proof} 
The label of a block in the Bose YXI-labeling is precisely the number of blocks preceding it in the given order. Therefore, 
\begin{align*}
\lb(B_{x,y,i}) &= \left|\left\{B_x \colon x \in [0,m-1]\right\}\right| + \left|\left\{B_{x,y,j} \colon 0 \leq j < i\right\}\right|
+ \left|\left\{B_{x',y,j} \colon 0 \leq x' < x,\ j \in [0,2]\right\}\right|\\
&\quad+ \left|\left\{B_{x',y',j} \colon 0 \leq x' < y' < y,\ j \in [0,2]\right\}\right|\\
&= m + i + 3x + 3\sum_{y'=1}^{y-1}y'\\
&= m + i + 3x + \dfrac{3}{2}y(y-1). 
\end{align*}
This proves the claimed result. 
\end{proof} 

\begin{lemma}
\label{lem:Bose_inverse}
For any $z \in [0,m-1]$, the set $P_z$ comprising all pairs of values $(x,y)$, such that $0 \leq x < y \leq m-1$, and $z = x\ob y$, is given as follows: 
\[
P_z = 
\begin{cases}
\{(x,2z-x) \colon x \in [0,z-1]\} \cup \left\{(x,2z-x+m) \colon x \in \left[2z+1,z+\frac{m-1}{2}\right]\right\},& \text{ if } z \leq \frac{m-1}{2},\\
\left\{(x,2z-x-m) \colon x \in \left[0,z-\frac{m+1}{2}\right]\right\} \cup \left\{(x,2z-x) \colon x \in \left[2z-m+1,z-1\right]\right\},& \text{ if } z \geq \frac{m+1}{2}.\\
\end{cases}
\] 
\end{lemma} 
\begin{proof} 
It is straightforward to verify that if $(x,y) \in P_z$, then $0 \leq x < y \leq m-1$ and $x \ob y = z$. Moreover, all points in the set are distinct. Since $|P_z| = (m-1)/2$, we conclude that the claim holds true. 
\end{proof} 

\begin{lemma} 
\label{lem:sB}
The sum of the labels of the blocks containing $(z,i)$, $z \in [0,m-1]$, $i \in [0,2]$, according to the Bose YXI-labeling, is given as follows:
\[
\lb(z,i) = 
\begin{cases}
\frac{5}{144}n^3 - \frac{6z+15}{144}n^2 - \frac{12z^2-84z-16i-8j-25}{48}n 
+ \frac{64z^3-84z^2-158z-16i-8j-25}{16},& \text{ for } z \leq \frac{m-1}{2},\\
\frac{1}{48}n^3 + \frac{42z-39}{144}n^2 - \frac{84z^2-108z-16i-8j-55}{48}n 
+ \frac{64z^3-84z^2-158z-16i-8j-25}{16},& \text{ for } z \geq \frac{m+1}{2},
\end{cases}
\] 
where $j \define (i - 1) \pmod 3 \in [0,2]$. 
\end{lemma} 
\begin{proof} 
The blocks containing $(z,i)$ include:
\begin{description}
	\item[Type 1:] $B_z = \{(z,0),(z,1),(z,2)\}$. 
	\item[Type 2:] $B_{z,y,i} = \{(z,i),(y,i),(z\ob y,(i+1) \pmod 3)\}$, for $z < y \leq m-1$.
	\item[Type 3:] $B_{x,z,i} = \{(x,i),(z,i),(x\ob z,(i+1) \pmod 3)\}$, for $0 \leq x < z$.
	\item[Type 4:] $B_{x,y,j} = \{(x,j),(y,j),(z,i)\}$, for $0 \leq x < y \leq m-1$ such that $x\ob y = z$ and $i=j+1\, (\mod \, 3)$.
\end{description}
 
Clearly, there is a unique block of Type~1 that contains $(z,i)$, namely $B_z$. According to Lemma~\ref{lem:Bose_Bx}, this block has label
\begin{equation} 
\label{eq:Bose_Type1}
\ell^{\text{T1}}_B(z,i) \define \lb(B_z) = m - 1 - z. 
\end{equation} 
The blocks $B_{z,y,i}$ of Type~2, where $y > z$, have a sum of labels that may be obtained by invoking Lemma~\ref{lem:Bose_Bxyi}:
\begin{equation} 
\label{eq:Bose_Type2}
\begin{aligned}
\ell^{\text{T2}}_B(z,i) &\define \sum_{y = z+1}^{m-1} \lb\big(B_{z,y,i}\big)
= \sum_{y = z+1}^{m-1} \Big(m + i + 3z + \dfrac{3}{2}y(y-1)\Big)\\
&= \dfrac{m^3}{2} - \dfrac{m^2}{2} + (2z+i)m - \dfrac{z^3+6z^2+5z+2iz+2i}{2}. 
\end{aligned}
\end{equation} 
The blocks $B_{x,z,i}$ of Type~3, where $x < z$, have the following sum of labels which again may be computed using Lemma~\ref{lem:Bose_Bxyi},
\begin{equation} 
\label{eq:Bose_Type3}
\begin{aligned}
\ell^{\text{T3}}_B(z,i) \define \sum_{x = 0}^{z-1} \lb\big(B_{x,z,i}\big)
= \sum_{x = 0}^{z-1} \Big(m + i + 3x + \dfrac{3}{2}z(z-1)\Big)
= zm + \Big(\dfrac{3z^3}{2}-\dfrac{3z}{2}+iz\Big). 
\end{aligned}
\end{equation} 
The sums of labels of blocks of Type~4 containing $(z,i)$ depends on whether $z \leq (m-1)/2$ or
$z \geq (m+1)/2$.\\ 
 
\textbf{Case 1: $0 \leq z \leq (m-1)/2$.} From Lemma~\ref{lem:Bose_inverse}, the sum of labels of blocks of Type~4 containing $(z,i)$ equals
\begin{equation}
\label{eq:Bose_Type4_Case1}
\begin{split}
\ell^{\text{T4}}_B(z,i) &\define \sum_{x = 0}^{z-1} \lb\big(B_{x,2z-x,j}\big)
+ \sum_{x = 2z+1}^{z+\frac{m-1}{2}} \lb\big(B_{x,2z-x+m,j}\big)\\
&= \sum_{x = 0}^{z-1}\Big(m+j+3x+\dfrac{3}{2}(2z-x)(2z-x-1)\Big)\\
&+ \sum_{x = 2z+1}^{z+\frac{m-1}{2}}\Big(m+j+3x+\dfrac{3}{2}(2z-x+m)(2z-x+m-1)\Big)\\
&= \dfrac{7}{16}m^3 - \dfrac{6z+7}{16}m^2 - \dfrac{12z^2-36z-8j-9}{16}m + 
\dfrac{48z^3-36z^2-78z-8j-9}{16}. 
\end{split}
\end{equation} 

\textbf{Case 2: $(m+1)/2 \leq z \leq m-1$.} By Lemma~\ref{lem:Bose_inverse}, the sum of labels of blocks of Type~4 containing $(z,i)$ equals
\begin{equation}
\label{eq:Bose_Type4_Case2}
\begin{split}
\ell^{\text{T4}}_B(z,i) &\define \sum_{x = 0}^{z-\frac{m+1}{2}} \lb\big(B_{x,2z-x-m,j}\big)
+ \sum_{x = 2z-m+1}^{z-1} \lb\big(B_{x,2z-x,j}\big)\\
&= \sum_{x = 0}^{z-\frac{m+1}{2}}\Big(m+j+3x+\dfrac{3}{2}(2z-x-m)(2z-x-m-1)\Big)\\
&+ \sum_{x = 2z-m+1}^{z-1}\Big(m+j+3x+\dfrac{3}{2}(2z-x)(2z-x-1)\Big)\\
&= \dfrac{1}{16}m^3 + \dfrac{42z-31}{16}m^2 - \dfrac{84z^2-60z-8j-39}{16}m + 
\dfrac{48z^3-36z^2-78z-8j-9}{16}.  
\end{split}
\end{equation} 

Summing up \eqref{eq:Bose_Type1}, \eqref{eq:Bose_Type2}, \eqref{eq:Bose_Type3}, and \eqref{eq:Bose_Type4_Case1} or \eqref{eq:Bose_Type4_Case2}, we obtain
\[
\begin{split}
\ell_B(z,i) &= \ell^{\text{T1}}_B(z,i) + \ell^{\text{T2}}_B(z,i) + \ell^{\text{T3}}_B(z,i)
+ \ell^{\text{T4}}_B(z,i)\\
&=
\begin{cases}
\frac{15}{16}m^3 - \frac{6z+15}{16}m^2 - \frac{12z^2-84z-16i-8j-25}{16}m 
+ \frac{64z^3-84z^2-158z-16i-8j-25}{16},& \text{ for } z \leq \frac{m-1}{2},\\
\frac{9}{16}m^3 + \frac{42z-39}{16}m^2 - \frac{84z^2-108z-16i-8j-55}{16}m 
+ \frac{64z^3-84z^2-158z-16i-8j-25}{16},& \text{ for } z \geq \frac{m+1}{2}. 
\end{cases}
\end{split}
\]
The proof follows by replacing $m$ with $n/3$ in the above formula.
\end{proof} 

We are now ready to prove Theorem~\ref{thm:Bose_YXI}. 

\begin{proof}[Proof of Theorem~\ref{thm:Bose_YXI}]
It suffices to show that $\ell_B(z,i) \geq f_{BYXI}(n)$ for every permissible value of $n \geq 27$ and for every $z \in [0,m-1]$, $i \in [0,2]$. Moreover, the equality holds when $i = 0,1$, and 
\[
z = 
\begin{cases}
(n-3)/12,& \text{ when } n/3 \equiv 1 \pmod 4,\\ 
(n-9)/12,& \text{ when } n/3 \equiv 3 \pmod 4.  
\end{cases}
\]
We first show that $\ell_B(z,i) > \ell_B\big(z-\frac{m+1}{2},i\big)$ for every $z \geq \frac{m+1}{2}$, and then proceed to prove that $\ell_B(z,i) \geq f_{BYXI}(n)$ for every $z \in \big[0,\frac{m-1}{2}\big]$ and $i \in [0,2]$. 

Suppose that $z = \frac{m+1}{2} + t = \frac{n+3}{6} + t$, where $t \in [0,\frac{m-1}{2}-1]$. The goal is to show that
$\ell_B(z,i) \geq \ell_B(t,i)$. From Lemma~\ref{lem:sB}, we have
\[
\begin{split}
\ell_B(z,i) &= \frac{1}{48}n^3 + \frac{42z-39}{144}n^2 - \frac{84z^2-108z-16i-8j-55}{48}n 
+ \frac{64z^3-84z^2-158z-16i-8j-25}{16}\\
&= \frac{17}{432}n^3 + \frac{2t-1}{48}n^2 + \frac{12t^2+36t+16i+8j-9}{48}n 
+ \frac{64t^3+12t^2-194t-16i-8j-117}{16}.\\
\end{split}
\] 
Also, from Lemma~\ref{lem:sB}, it holds that
\[
\ell_B(t,i) = \frac{5}{144}n^3 - \frac{6t+15}{144}n^2 - \frac{12t^2-84t-16i-8j-25}{48}n 
+ \frac{64t^3-84t^2-158t-16i-8j-25}{16}.
\]
Therefore, we deduce that
\[
\ell_B(z,i) - \ell_B(t,i) = 
\dfrac{1}{216}\Big((108n+1296)t^2 + 18(n^2-12n-27)t + (n^3+18n^2 - 153n - 1242)\Big) > 0,
\]
since $t \geq 0$ and the coefficients of all powers of $t$ are positive for $n \geq 27$. Thus, for every $z \geq \frac{m+1}{2}$ and $i \in [0,2]$, we have $\ell_B(z,i) > \ell_B\big(z-\frac{m+1}{2},i\big)$. 

It remains to prove that $\ell_B(z,i) \geq f_{BYXI}(n)$ for every $z \in \big[0,\frac{m-1}{2}\big]$ and $i \in [0,2]$. To this end, let 
\[
g(n,z,i) \define \ell_B(z,i) - f_{BYXI}(n).
\]
By its definition, $f_{BYXI}(n)$ depends on the congruence class of $m=n/3 \pmod 4$. Therefore, we separately consider two different cases.\\

\textbf{Case 1: $n/3 \equiv 1 \pmod 4$ and $n \geq 27$.} We have
\[
\begin{split}
g(n,z,i) &= 
\Bigg(\dfrac{5}{144}n^3 - \dfrac{6z+15}{144}n^2 - \dfrac{12z^2-84z-16i-8j-25}{48}n + \dfrac{64z^3-84z^2-158z-16i-8j-25}{16}\Bigg)\\
&- \Bigg( \dfrac{55}{1728}n^3 + \dfrac{1}{192}n^2 - \dfrac{9}{64}n - \dfrac{31}{64}\Bigg)\\
&= 
\Bigg(\dfrac{5}{144}n^3 - \dfrac{6z+15}{144}n^2 - \dfrac{12z^2-84z-41}{48}n + \dfrac{64z^3-84z^2-158z-41}{16}\Bigg)\\
&- \Bigg( \dfrac{55}{1728}n^3 + \dfrac{1}{192}n^2 - \dfrac{9}{64}n - \dfrac{31}{64}\Bigg) + \Bigg(\dfrac{16i+8j-16}{48}n-\dfrac{16i+8j-16}{16}\Bigg). 
\end{split}
\]
The last term in the above sum is nonnegative, since
\[
\dfrac{16i+8j-16}{48}n-\dfrac{16i+8j-16}{16} 
= \dfrac{1}{48}(16i+8j-16)(n-3) \geq 0
\]
for every $i,j \in [0,2]$, $j \equiv (i - 1) \pmod 3$, and $n \geq 3$. 
For $n \geq 27$, the equality holds if and only if $i = 1$ and $j = 0$ or $i = 0$ and $j = 2$. Therefore, 
\[
\begin{split}
g(n,z,i) 
&\geq 
\Bigg(\dfrac{5}{144}n^3 - \dfrac{6z+15}{144}n^2 - \dfrac{12z^2-84z-41}{48}n + \dfrac{64z^3-84z^2-158z-41}{16}\Bigg)\\
&- \Bigg( \dfrac{55}{1728}n^3 + \dfrac{1}{192}n^2 - \dfrac{9}{64}n - \dfrac{31}{64}\Bigg)\\
&= \Big(z-z^*\Big)h(n,z),
\end{split}
\]
where
\[
z^* \define \dfrac{n-3}{12}
\]
and
\[
h(n,z) \define 4z^2 + \dfrac{n-75}{12}z - \dfrac{5n^2-174n+1197}{144}.
\]
In order to prove that $g(n,z,i) \geq 0$, it suffices to show that $z-z^*$ and $h(n,z)$ do not have opposite signs for $z \in [0,n/3-1]$ and $n \geq 27$. We hence proceed to analyze the sign of the quadratic polynomial $h(n,z)$ in $z$. 
First, we compute
\[
\Delta(n) = \Big(\dfrac{n-75}{12}\Big)^2 + 16\dfrac{5n^2-174n+1197}{144}
= \dfrac{9n^2-326n+2753}{16}. 
\]
It is easy to see that $\Delta(n) > 0$ whenever $n \geq 23$. As we assume that $n \geq 27$, $h(n,z)$ has two distinct roots
\[
z_1 = \dfrac{-\frac{1}{12}(n-75)-\sqrt{\Delta(n)}}{8},
\quad z_2 = \dfrac{-\frac{1}{12}(n-75)+\sqrt{\Delta(n)}}{8}. 
\]
It can be verified that $z_1 < 0$ when $n \geq 26$.
Since $h(n,z) \geq 0$ when $z \geq z_2$
and $h(n,z) \leq 0$ when $0 \leq z \leq z_2$, for $(z-z^*)h(n,z)$ to be nonnegative, it suffices to show that 
\[
z^*-1 \leq z_2 \leq z^*.
\] 
Then, if $z \geq z^* \geq z_2$, we have $z-z^* \geq 0$ and $h(n,z) \geq 0$, while if $z \leq z^*-1 \leq z_2$, we have $z-z^* < 0$ and $h(n,z) \leq 0$. Both cases lead to $(z-z^*)h(n,z) \geq 0$ as $z \in [0,\frac{m-1}{2}]$ is a nonnegative integer. 
We hence have 
\[
z_2 \leq z^* \Longleftrightarrow \dfrac{-\frac{1}{12}(n-75)+\sqrt{\Delta(n)}}{8} \leq \dfrac{n-3}{12} 
\Longleftrightarrow \sqrt{9n^2-326n+2753} \leq 3n-33
\Longleftrightarrow n \geq 13,
\]
which is true because we assumed that $n \geq 27$. Similarly, 
\[
z^*-1 \leq z_2 \Longleftrightarrow \dfrac{n-3}{12} - 1 \leq \dfrac{-\frac{1}{12}(n-75)+\sqrt{\Delta(n)}}{8}
\Longleftrightarrow 3n-65 \leq \sqrt{9n^2-326n+2753}
\Longleftrightarrow n \geq 23.
\]
Therefore, as long as $n \geq 27$ (the smallest valid value of $n$ greater than or equal to $26$), it holds that $g(n,z,i) \geq (z-z^*)h(n,z) \geq 0$, which implies that $\ell_B(z,i) \geq f_{BYXI}(n)$, the inequality we set out to prove.
It is clear that when $z = \frac{m-1}{4} = \frac{n-3}{12}$ and $i = 0,1$, equality is met, that is, $\ell_B(z,i) = f_{BYXI}(n)$. Thus, in this case, the dual min-sum of the $\stsn$ in the Bose Construction together with the Bose YXI-labeling of blocks is precisely $f_{BYXI}(n)$.\\

\textbf{Case 2: $n/3 \equiv 3 \pmod 4$ and $n \geq 33$.} Similarly as in the previous case, we have
\[
\begin{split}
g(n,z,i) 
&= \Bigg(\dfrac{5}{144}n^3 - \dfrac{6z+15}{144}n^2 - \dfrac{12z^2-84z-41}{48}n + \dfrac{64z^3-84z^2-158z-41}{16}\Bigg)\\
&- \Bigg( \dfrac{55}{1728}n^3 + \dfrac{1}{192}n^2 - \dfrac{13}{64}n + \dfrac{13}{64}\Bigg) + \Bigg(\dfrac{16i+8j-16}{48}n-\dfrac{16i+8j-16}{16}\Bigg)\\
&\geq (z-z^{**})k(n,z),
\end{split}
\]
where $z^{**} \define \frac{n-9}{12}$ and $k(n,z) \define 4z^2 + \frac{n-99}{12}z
- \frac{5n^2-144n+531}{144}$. As long as $n \geq 19$, the quadratic polynomial $k(n,z)$ has two roots
\[
z_1 = \dfrac{-\frac{1}{12}(n-99)-\frac{1}{4}\sqrt{9n^2-278n+2033}}{8},
\quad z_2 = \dfrac{-\frac{1}{12}(n-99)+\frac{1}{4}\sqrt{9n^2-278n+2033}}{8}.
\]
It can be easily verified that when $n \geq 25$, we have $z_1 < 0$. Also, we have
\[
z^{**} \leq z_2 \leq z^{**} + 1,
\]
where the first inequality holds when $n \geq 19$ and the second inequality holds when $n \geq 11$. The smallest valid $n$ greater than or equal to $25$ is $n = 33$. For $n \geq 33$, if $z \geq z^{**}+1 \geq z_2$, then $z-z^{**} > 0$ and $k(n,z) \geq 0$, and if $z \leq z^{**}$, then $z-z^{**} < 0$ and $k(n,z) \leq 0$. Both cases lead to $g(n,z,i) \geq (z-z^{**})k(n,z) \geq 0$, which implies that $\ell_B(z,i) \geq f_{BYXI}$, as desired. The equality is obtained when $z = \frac{m-3}{4} = \frac{n-9}{12}$ and $i = 0,1$. This completes the proof of Theorem~\ref{thm:Bose_YXI}. 
\end{proof} 


\begin{proposition}
The Bose $\stsn$, in which the blocks are labeled according to the Bose YXI-labeling, has dual max-sum equal to
\[
g_{BYXI}(n) = \dfrac{31}{432}n^3 - \dfrac{9}{16}n^2 + \dfrac{35}{16}n - \dfrac{55}{16}, \quad n \geq 15.
\]
Therefore, the dual difference-sum and the dual ratio-sum of this construction and block labeling are $\O\big(23n^3/1728\big)$ and $\O(124/55)$, respectively. 
\begin{proof}
The proof that establishes the dual max-sum is similar to the proof that establishes the dual min-sum and hence is omitted. Note that it suffices to show that $g_{BYXI}(n) - \ell_B(z,i) \geq 0$ for all $m-1 \geq z \geq \frac{m+1}{2}$, $i \in [0,2]$, $n \geq 15$, and that equality holds if and only if $(z,i) = (m-1,2)$.
\end{proof}
\end{proposition}

The following more natural block ordering of the $\stsn$ produced by the Bose Construction gives a smaller dual min-sum compared to the Bose YXI-labeling. We omit the proof.

\begin{theorem}
\label{thm:Bose_dual}
The dual of the Bose $\stsn$, where the blocks are labeled from $0$ to $\frac{n(n-1)}{6}-1$ in the following order: $B_x$, $x = 0,1,\ldots, m-1$, followed by $B_{x,y,i}$, $x = 0,1,\ldots,m-2$, $y = x+1,x+2,\ldots,m-1$, $i = 0,1,2$, has min-sum equal to
\[
f_B(n) = \dfrac{5}{432}n^3 + \dfrac{19}{48}n^2 - \dfrac{133}{48}n + \dfrac{71}{16},\quad n \geq 9. 
\] 
\end{theorem} 

We state next a similar result for the dual of the Skolem $\stsn$ (Theorem~\ref{thm:Skolem_YXI}). 
As the full proof is tedious and mostly follows along the same lines as the corresponding proof of Theorem~\ref{thm:Bose_YXI}, 
we only present the key lemmas and important formulas without including all details. 

\begin{theorem}
\label{thm:Skolem_YXI}
The dual of the Skolem $\stsn$, where the blocks are labeled from $0$ to $\frac{n(n-1)}{6}-1$ in the following order, referred to as the Skolem YXI-labeling: $B_x$, $x = 0,1,\ldots, m/2-1$, followed by $B_{x,y,i}$, $y = 1,2,\ldots,m-1$, $x = 0,1,\ldots,y-1$, $i = 0,1,2$, followed by $B_{x,i}$, $x = 0,1,\ldots,m/2-1$, $i = 0,1,2$, where $m = (n-1)/3$, and the blocks $B_{x,i}$ are labeled as 
\[
\dfrac{n(n-1)}{6} - \dfrac{n-1}{2} + 
\begin{cases}
(n-1)/6 + 2x, &i = 0,\\
x, &i = 1,\\
(n-1)/6 + 2x + 1, &i = 2,
\end{cases}
\]
has min-sum equal to
\[
f_{SYXI}(n) = 
\begin{cases}
\dfrac{55}{1728}n^3 - \dfrac{31}{576}n^2 - \dfrac{137}{576}n - \dfrac{1279}{1728},& \text{ if } (n-1)/3 \equiv 0 \pmod 4, \ n \geq 13, \\\\
\dfrac{55}{1728}n^3 - \dfrac{31}{576}n^2 - \dfrac{173}{576}n + \dfrac{1421}{1728},& \text{ if } (n-1)/3 \equiv 2 \pmod 4, \ n \geq 7.
\end{cases}
\] 
\end{theorem}

\begin{lemma} 
\label{lem:Skolem_Bx}
In the Skolem YXI-labeling, the labels of the blocks $B_x$, $x \in [0,m/2-1]$, equal
\[
\ls(B_x) = x.
\]
\end{lemma} 

\begin{lemma} 
\label{lem:Skolem_Bxyi}
In the Skolem YXI-labeling, the labels of the blocks $B_{x,y,i}$, $0 \leq x < y \leq m-1$, $i \in [0,2]$, equal
\[
\ls(B_{x,y,i}) = \dfrac{m}{2} + i + 3x + \dfrac{3}{2}y(y-1).
\]
\end{lemma} 

Note that the labels of the blocks $B_{x,i}$, $x \in [0,m/2-1]$, $i \in [0,2]$, are
given explicitly in the statement of Theorem~\ref{thm:Skolem_YXI}. 

\begin{lemma} 
\label{lem:Skolem_inverse}
For any $z \in [0,m-1]$, the set $Q_z = \{(x,y) \colon 0 \leq x < y \leq m-1, z = x \os y\}$ has the explicit form
\[
Q_z = 
\begin{cases}
\{(x,2z-x) \colon x \in [0,z-1]\} \cup \left\{(x,2z-x+m) \colon x \in \left[2z+1,z+\frac{m}{2}-1\right]\right\},& \text{ if } z < \frac{m}{2},\\
\left\{(x,2z-x-m+1) \colon x \in \left[0,z-\frac{m}{2}\right]\right\} \cup \left\{(x,2z-x+1) \colon x \in \left[2z-m+2,z\right]\right\},& \text{ if } z \geq \frac{m}{2}.\\
\end{cases}
\]
\end{lemma} 

\begin{lemma} 
\label{lem:sS}
The sum of the labels of the blocks containing $(z,i)$, $z \in [0,m-1]$, $i \in [0,2]$, according to the Skolem YXI-labeling, is given as follows. 

For $z < \frac{m}{2}$, 
\[
\begin{split}
\ls(z,i) =\ 
& \dfrac{5}{144}n^3 - \dfrac{2z+7}{48}n^2 - \dfrac{36z^2-228z-3-48i-24j}{144}n\\ 
&+ \dfrac{576z^3-828z^2-1518z+13-192i-168j}{144}
+ \begin{cases}
(n-1)/6 + 2z, &j = 0,\\
z, &j = 1,\\
(n-1)/6 + 2z + 1, &j = 2,
\end{cases}
\end{split}
\]
where $j \define (i - 1) \pmod 3 \in [0,2]$. 

For $z \geq \frac{m}{2}$, 
\[
\begin{split}
\ls(z,i) =\ 
& \dfrac{1}{48}n^3 + \dfrac{14z-5}{48}n^2 - \dfrac{84z^2+4z-47-16i-8j}{48}n\\ 
&+ \dfrac{192z^3+84z^2-346z-187-64i-8j}{48}
+ \begin{cases}
2z, &i = 0,\\
z, &i = 1,\\
2z + 1, &i = 2,
\end{cases}
\end{split}
\]
where again $j \define (i - 1) \pmod 3 \in [0,2]$.
\end{lemma} 
\begin{proof}[Sketch of the Proof]
The blocks containing $(z,i)$ include:
\begin{description}
	\item[Type 1:] $B_z = \{(z,0),(z,1),(z,2)\}$, if $0 \leq z < \frac{m}{2}$.  
	\item[Type 2:] $B_{z,y,i} = \{(z,i),(y,i),(z\os y,(i+1) \pmod 3)\}$, for $0 \leq z < y \leq m-1$.
	\item[Type 3:] $B_{x,z,i} = \{(x,i),(z,i),(x\os z,(i+1) \pmod 3)\}$, for $0 \leq x < z \leq m-1$.
	\item[Type 4:] $B_{x,y,j} = \{(x,j),(y,j),(z,i)\}$, for $0 \leq x < y \leq m-1$ such that $x\os y = z$.
	\item[Type 5:] $B_{z,j} = \{\infty, (z+\frac{m}{2},j), (z,i)\}$, where $j \define (i - 1) \pmod 3 \in [0,2]$, where $z < \frac{m}{2}$.  
	\item[Type 6:] $B_{z-\frac{m}{2},i} = \{\infty,(z,i),(z-\frac{m}{2},(i+1) \pmod 3)\}$, for $z \geq \frac{m}{2}$.
\end{description}
Based on Lemmas~\ref{lem:Skolem_Bx},~\ref{lem:Skolem_Bxyi}, and~\ref{lem:Skolem_inverse}, one can compute the sums of labels of blocks of each type as follows. 

\nin\textbf{Type 1:}
\begin{equation} 
\label{eq:Skolem_Type1}
\ell^{\text{T1}}_S(z,i) \define \lb(B_z) = z. 
\end{equation} 
\textbf{Type 2:}
\begin{equation} 
\label{eq:Skolem_Type2}
\begin{aligned}
\ell^{\text{T2}}_S(z,i) &\define \sum_{y = z+1}^{m-1} \ls\big(B_{z,y,i}\big)
= \sum_{y = z+1}^{m-1} \Big(\dfrac{m}{2} + i + 3z + \dfrac{3}{2}y(y-1)\Big)\\
&= \dfrac{1}{2}\Big(m^3 - 2m^2 + (5z+2i+1)m - (z^3+6z^2+5z+2iz+2i)\Big). 
\end{aligned}
\end{equation} 
\textbf{Type 3:}
\begin{equation} 
\label{eq:Skolem_Type3}
\begin{aligned}
\ell^{\text{T3}}_S(z,i) &\define \sum_{x = 0}^{z-1} \ls\big(B_{x,z,i}\big)
= \sum_{x = 0}^{z-1} \Big(\dfrac{m}{2} + i + 3x + \dfrac{3}{2}z(z-1)\Big)
= \dfrac{z}{2}m + \Big(\dfrac{3z^3}{2}-\dfrac{3z}{2}+iz\Big). 
\end{aligned}
\end{equation} 
\textbf{Type 4:}
The sums of labels of blocks of Type~4 containing $(z,i)$ depends on whether $z < \frac{m}{2}$ or $z \geq \frac{m}{2}$. For $z < \frac{m}{2}$, we have
\begin{equation}
\label{eq:Skolem_Type4_Case1}
\begin{split}
\ell^{\text{T4}}_S(z,i) &\define \sum_{x = 0}^{z-1} \ls\big(B_{x,2z-x,j}\big)
+ \sum_{x = 2z+1}^{z+\frac{m}{2}-1} \ls\big(B_{x,2z-x+m,j}\big)
= \sum_{x = 0}^{z-1}\Big(\dfrac{m}{2}+j+3x+\dfrac{3}{2}(2z-x)(2z-x-1)\Big)\\
&= \dfrac{7}{16}m^3 - \dfrac{3z+7}{8}m^2 - \dfrac{3z^2-6z-2j}{4}m + 
(3z^3-3z^2-6z-j). 
\end{split}
\end{equation} 
For $z \geq \frac{m}{2}$, we have
\begin{equation}
\label{eq:Skolem_Type4_Case2}
\begin{split}
\ell^{\text{T4}}_S(z,i) &\define \sum_{x = 0}^{z-\frac{m}{2}} \ls\big(B_{x,2z-x-m+1,j}\big)
+ \sum_{x = 2z-m+2}^{z} \ls\big(B_{x,2z-x+1,j}\big)\\
&= \sum_{x = 0}^{z-\frac{m}{2}}\Big(\dfrac{m}{2}+j+3x+\dfrac{3}{2}(2z-x-m+1)(2z-x-m)\Big)\\
&+ \sum_{x = 2z-m+2}^{z}\Big(\dfrac{m}{2}+j+3x+\dfrac{3}{2}(2z-x+1)(2z-x)\Big)\\
&= \dfrac{1}{16}m^3 + \dfrac{7(3z-1)}{8}m^2 - \dfrac{21z^2+6z-2j-14}{4}m + 
3(z-1)(z+1)^2.  
\end{split}
\end{equation} 
\textbf{Type 5:}
\begin{equation}
\label{eq:Skolem_Type5}
\ell^{\text{T4}}_S(z,i) \define \ls(B_{z,j}) = 
\dfrac{n(n-1)}{6} - \dfrac{n-1}{2} + 
\begin{cases}
(n-1)/6 + 2z, &j = 0,\\
z, &j = 1,\\
(n-1)/6 + 2z + 1, &j = 2.
\end{cases}
\end{equation} 
\textbf{Type 6:}
\begin{equation}
\begin{split}
\label{eq:Skolem_Type6}
\ell^{\text{T4}}_S(z,i) &\define \ls(B_{z-\frac{m}{2},j}) = 
\dfrac{n(n-1)}{6} + 
\begin{cases}
m/2 + 2(z-m/2), &i = 0\\
z-m/2, &i = 1\\
m/2 + 2(z-m/2) + 1, &i = 2
\end{cases}\\
&=
\dfrac{(n-1)(n-4)}{6} + 
\begin{cases}
2z, &i = 0,\\
z, &i = 1,\\
2z + 1, &i = 2.
\end{cases}
\end{split}
\end{equation} 
Finally, to compute $\ls(z,i)$, we sum up \eqref{eq:Skolem_Type1}, \eqref{eq:Skolem_Type2}, \eqref{eq:Skolem_Type3}, \eqref{eq:Skolem_Type4_Case1}, and \eqref{eq:Skolem_Type5} for $z < m/2$, and \eqref{eq:Skolem_Type2}, \eqref{eq:Skolem_Type3}, \eqref{eq:Skolem_Type4_Case2}, and \eqref{eq:Skolem_Type6} for $z \geq m/2$. The proof follows by replacing $m$ with $(n-1)/3$ in the obtained sums. 
\end{proof} 

\begin{proof}[Proof of Theorem~\ref{thm:Skolem_YXI} (Sketch)] 
The proof follows along the same lines as the corresponding proof of Theorem~\ref{thm:Bose_YXI}, with some slightly more complicated technical details. Note that we can ignore the point $\infty$ because the $\frac{n-1}{2}$ blocks containing this point have the largest possible labels from $\frac{n(n-1)}{6}-\frac{n-1}{2}$ to $\frac{n(n-1)}{6}-1$, the sum of which is clearly larger than the sum of the labels of the blocks containing any other point $(z,i)$. It also suffices to show that $\ls(z,i) \geq f_{SYXI}(n)$ for every $z < \frac{m}{2}$ and $i \in [0,2]$.

In order to evaluate the difference $\ls(z,i) - f_{SYXI}(n)$, we need to eliminate the part of $\ls(z,i)$ that involves $i$ and $j$. More precisely, a lower bound on the expression 
\[
\dfrac{48i+24j}{144}n
- \dfrac{192i+168j}{144}
+ \begin{cases}
(n-1)/6 + 2z, &j = 0,\\
z, &j = 1,\\
(n-1)/6 + 2z + 1, &j = 2,
\end{cases}
\]
that depends only on $n$ and $z$ must be established. We can simplify and evaluate this expression as follows
\[
\dfrac{2i+j}{6}(n-4) - \dfrac{j}{2} 
+ \begin{cases}
(n-1)/6 + 2z, &j = 0\\
z, &j = 1\\
(n-1)/6 + 2z + 1, &j = 2
\end{cases}
\qquad \geq \dfrac{n-4}{3} + \dfrac{n-1}{6} + 2z,
\]
where the inequality holds when $j=0,2$, or equivalently, when $i = (j+1) \pmod 3 \in \{0,1\}$. To prove the inequality when $j = 1$, we use the fact that $n \geq 7$ and $z \leq \frac{m}{2}-1 = \frac{n-7}{6}$. Therefore, for $z < \frac{m}{2}$ it holds that 
\[
\begin{split}
\ls(z,i) &\geq \dfrac{5}{144}n^3 - \dfrac{2z+7}{48}n^2 - \dfrac{36z^2-228z-3}{144}n
+ \dfrac{576z^3-828z^2-1518z+13}{144}\\
&\ + \Big(\dfrac{n-4}{3} + \dfrac{n-1}{6} + 2z\Big)\\
&= \dfrac{5}{144}n^3 - \dfrac{2z+7}{48}n^2 - \dfrac{36z^2-228z-75}{144}n
+ \dfrac{576z^3-828z^2-1230z-203}{144}.
\end{split}
\] 
Hence, for $(n-1)/3 \equiv 0 \pmod 4$, we have
\[
\ls(z,i) - f_{SYXI}(n) \geq \Big(z - \dfrac{n-1}{12}\Big)
\Big( 4z^2 + \dfrac{n-73}{12}z - \dfrac{5n^2-144n+1157}{144} \Big),
\]
and for $(n-1)/3 \equiv 2 \pmod 4$, we have
\[
\ls(z,i) - f_{SYXI}(n) \geq \Big(z - \dfrac{n-7}{12}\Big)
\Big( 4z^2 + \dfrac{n-97}{12}z - \dfrac{5n^2-124n+551}{144} \Big). 
\]
The same argument as in the proof of Theorem~\ref{thm:Bose_YXI} may be used to establish that $\ls(z,i) - f_{SYXI}(n) \geq 0$, for every valid $n \geq 7$. Moreover, equality holds when $z = \frac{n-1}{12}$ if $(n-1)/3 \equiv 0 \pmod 4$ and when $z = \frac{n-7}{12}$ if $(n-1)/3 \equiv 2 \pmod 4$ (and $i = 0,1$). This completes the proof of the theorem. 
\end{proof} 

It is straightforward to see that the dual max-sum of the Skolem $\stsn$ with the Skolem YXI-labeling equals
\[
g_{SYXI}(n) = \frac{1}{12}n^3 - \frac{7}{24}n^2 + \frac{1}{12}n + \frac{1}{8},
\] 
and that it corresponds to the sum of the labels of the blocks containing the point $\infty$.
As a consequence, the dual difference-sum and the dual ratio-sum are $\O\big(89n^3/1728\big)$ and $\O(144/55)$, respectively. 

The natural block ordering determined by the Skolem Construction produces an $\stsn$ with smaller dual min-sum compared to the Skolem YXI-labeling. The result is described in Theorem~\ref{thm:Skolem_dual}. We again omit the proof since is mostly involves algebraic manipulations. 

\begin{theorem}
\label{thm:Skolem_dual}
The dual of the Skolem $\stsn$, where the blocks are labeled from $0$ to $\frac{n(n-1)}{6}-1$ in the following order: $B_x$, $x = 0,1,\ldots, m/2-1$, followed by $B_{x,y,i}$, $x = 0,1,\ldots,m-2$, $y = x+1,x+2,\ldots,m-1$, $i = 0,1,2$, followed by $B_{x,i}$, $x = 0,1,\ldots,m/2-1$, $i = 0,1,2$, where $m = (n-1)/3$, has min-sum equal to
\[
f_S(n) = \dfrac{5}{432}n^3 + \dfrac{55}{144}n^2 - \dfrac{511}{144}n + \dfrac{3523}{432}, \quad n \geq 7.  
\]  
\end{theorem} 

\section{Concluding Remarks and Open Problems} \label{sec:conclusions}

Motivated by distributed storage codes and access-balancing issues that accompany them, we introduced a new family of problems in combinatorial design theory pertaining to sum-constrained Steiner systems. The constraints on the sums of points in the blocks of a design, or similar constraints for the duals of the design, were addressed by permuting (relabeling) the points and blocks of two classes of Steiner systems, the Bose and Skolem triple systems. Optimal labelings achieving upper bounds were identified for the designs, while near-optimal labelings were constructed for dual designs. 

Many open problems in this new area of design theory remain, including:
\begin{enumerate}
\item Establishing which nonisomorphic Steiner triple systems have MaxMinSum labelings.
\item Tightening the bounds on difference-sum and ratio-sum in Section~\ref{sec:prelim} and constructing labelings that attain the bounds. 
\item Improving the block labelings for dual designs. 
\item Deriving extensions of the presented results for Steiner systems with block sizes $k>3$, as for example the quadruple systems of~\cite{hanani1960quadruple}, and for block designs in general. 
\end{enumerate}

\section*{Acknowledgments}
The work was supported in part by the NSF grants CCF 15-26875 and 4101-38050, as well as by 
NIH BD2K grant U01CA198943-02S1. The authors are grateful to one of the anonymous reviewers for running extensive simulations that suggest that the max-sum of the Bose mapping equals $8n/3-9$ and that the max-sum of the Skolem mapping equals $(8n-17)/3$ for sufficiently large $n$. They also acknowledge many useful discussions with Charles Colbourn.

\bibliographystyle{plain}
\bibliography{MinSumDesigns}

\begin{thebibliography}{10}

\bibitem{Bose1939}
R.~C. Bose.
\newblock On the construction of balanced incomplete block designs.
\newblock {\em Annals of Eugenics}, 9:353--399, 1939.

\bibitem{BoseNairPBIBD}
R.~C. Bose and K.~R. Nair.
\newblock Partially balanced incomplete block designs.
\newblock {\em Sankhy\={a}: The Indian Journal of Statistics}, 4(3):337--372,
  1939.

\bibitem{BreslauWebCaching}
L.~Breslau, P.~Cao, L.~Fan, G.~Phillips, and S.~Shenker.
\newblock Web caching and {Z}ipf-like distributions: evidence and implications.
\newblock In {\em IEEE International Conference on Computer Communications
  (INFOCOM)}, volume~1, pages 126--134, 1999.

\bibitem{cherkasova2004analysis}
L.~Cherkasova and M.~Gupta.
\newblock Analysis of enterprise media server workloads: access patterns,
  locality, content evolution, and rates of change.
\newblock {\em IEEE/ACM Transactions on Networking}, 12(5):781--794, 2004.

\bibitem{colbourn2006handbook}
C.~J. Colbourn and J.~H. Dinitz.
\newblock {\em Handbook of combinatorial designs}.
\newblock CRC Press, 2006.

\bibitem{colbourn1999triple}
C.~J. Colbourn and A.~Rosa.
\newblock {\em Triple systems}.
\newblock Oxford University Press, 1999.

\bibitem{dimakis2010network}
A.~G. Dimakis, B.~Godfrey, Y.~Wu, M.~J. Wainwright, and K.~Ramchandran.
\newblock Network coding for distributed storage systems.
\newblock {\em IEEE Transactions on Information Theory}, 56(9):4539--4551,
  2010.

\bibitem{el2010fractional}
S.~El-Rouayheb and K.~Ramchandran.
\newblock Fractional repetition codes for repair in distributed storage
  systems.
\newblock In {\em Proceedings of the 48th Annual Allerton Conference on
  Communication, Control, and Computing (Allerton)}, pages 1510--1517. IEEE,
  2010.

\bibitem{fazeli2015codes}
A.~Fazeli, A.~Vardy, and E.~Yaakobi.
\newblock Codes for distributed {PIR} with low storage overhead.
\newblock In {\em Proceedings of the IEEE International Symposium on
  Information Theory (ISIT)}, pages 2852--2856. IEEE, 2015.

\bibitem{Hanani1960}
H.~Hanani.
\newblock A note on {S}teiner triple systems.
\newblock {\em Mathematica Scandinavica}, 8:154--156, 1960.

\bibitem{hanani1960quadruple}
H.~Hanani.
\newblock On quadruple systems.
\newblock {\em Canadian Journal of Mathematics}, 12(2):145--157, 1960.

\bibitem{joshi2014delay}
G.~Joshi, Y.~Liu, and E.~Soljanin.
\newblock On the delay-storage trade-off in content download from coded
  distributed storage systems.
\newblock {\em IEEE Journal on Selected Areas in Communications},
  32(5):989--997, 2014.

\bibitem{Kirkman1847}
T.~P. Kirkman.
\newblock On a problem in combinations.
\newblock {\em Cambridge and Dublin Mathematical Journal 2}, pages 191--204,
  1847.

\bibitem{leong2012distributed}
D.~Leong, A.~G. Dimakis, and T.~Ho.
\newblock Distributed storage allocations.
\newblock {\em IEEE Transactions on Information Theory}, 58(7):4733--4752,
  2012.

\bibitem{LindnerRodger1997}
C.~C. Lindner and C.~A. Rodger.
\newblock {\em Design Theory}.
\newblock CRC Press, 1997.

\bibitem{Pawar_etal2011}
S.~Pawar, N.~Noorshams, S.~El Rouayheb, and K.~Ramchandran.
\newblock {DRESS} codes for the storage cloud: {S}imple randomized
  constructions.
\newblock In {\em Proceedings of the IEEE International Symposium on
  Information Theory (ISIT)}, pages 2338--2342, 2011.

\bibitem{silberstein2015optimal}
N.~Silberstein and T.~Etzion.
\newblock Optimal fractional repetition codes based on graphs and designs.
\newblock {\em IEEE Transactions on Information Theory}, 61(8):4164--4180,
  2015.

\bibitem{silberstein2016optimal}
N.~Silberstein and A.~G{\'a}l.
\newblock Optimal combinatorial batch codes based on block designs.
\newblock {\em Designs, Codes and Cryptography}, 78(2):409--424, 2016.

\bibitem{Skolem1958}
T.~Skolem.
\newblock Some remarks on the triple systems of {S}teiner.
\newblock {\em Mathematica Scandinavica}, 6:273--280, 1958.

\bibitem{Stinson2004}
D.~R. Stinson.
\newblock {\em Combinatorial Designs: Constructions and Analysis}.
\newblock Springer, New York, 2004.

\bibitem{vasic2004combinatorial}
B.~Vasic and O.~Milenkovic.
\newblock Combinatorial constructions of low-density parity-check codes for
  iterative decoding.
\newblock {\em IEEE Transactions on information theory}, 50(6):1156--1176,
  2004.

\end{thebibliography}

\end{document}